\documentclass[10pt]{article}

\usepackage{color}
\usepackage[table,xcdraw,dvipsnames]{xcolor}
\usepackage{amsmath}
\usepackage{bookmark}
\usepackage{mathrsfs}
\usepackage{amsfonts}
\usepackage{enumitem}
\usepackage{amsthm,mathabx}
\usepackage{nicefrac}
\usepackage{amssymb}
\usepackage{graphicx,float}
\usepackage{cite}
\usepackage{soul}
\usepackage{arydshln}
\usepackage{makecell}
\usepackage{caption}
\usepackage{multicol}

\usepackage{booktabs}
\usepackage{nicematrix}
\usepackage{stackengine}

\definecolor{Nblue}{HTML}{0080ff}
\definecolor{NewGreen}{rgb}{0, 0.501, 0}
\definecolor{Red1}{rgb}{0.858, 0.188, 0.478}
\usepackage{hyperref}
\hypersetup{
	colorlinks=true,
	citecolor=Nblue,
	linkcolor=NewGreen,
	filecolor=magenta,      
	urlcolor=blue
}
\usepackage[left=1.5cm,right=1.5cm,top=1.5cm,bottom=2cm]{geometry}
\usepackage[nameinlink]{cleveref}
\definecolor{BlueNice1}{rgb}{0.243137, 0.396078, 0.709804}
\definecolor{BlueNice2}{rgb}{0.356863, 0.588235, 0.815686}
\definecolor{RedNice1}{rgb}{0.8, 0.262745, 0.231373}
\definecolor{RedNice2}{rgb}{0.67451, 0.207843, 0.184314}
\definecolor{GreenNice1}{rgb}{0.0980392, 0.631373, 0.152941}
\usepackage{amsthm}
\newtheorem{lemma}{Lemma}

\newtheorem{example}{Example}
\newtheorem{theorem}{Theorem}

\newtheorem{remark}{Remark}

\newtheoremstyle{claim}
{\topsep}
{\topsep}
{}
{}
{\itshape}
{.}
{.5em}
{\thmname{#1}\thmnumber{ #2}\thmnote{ (#3)}}
\theoremstyle{claim}

\usepackage{environ}%

\NewEnviron{neweq}{%
	\begingroup
	\allowdisplaybreaks
	\begin{align}
		\begin{split}
			\BODY
		\end{split}
	\end{align}
	\endgroup}

\NewEnviron{neweq_non}{%
	\begingroup
	\allowdisplaybreaks
	\begin{align*}
		\BODY
	\end{align*}
	\endgroup}
\usepackage{titlesec} 
\titleformat{\subsection}[runin]
{\normalfont\large\bfseries}{\thesubsection}{1em}{}
\titleformat*{\subsection}{\normalfont\bfseries}

\titleformat{\subsubsection}[runin]
{\normalfont\large\bfseries}{\thesubsubsection}{1em}{}
\titleformat*{\subsubsection}{\normalfont\bfseries}

\def\gavg{\tau_{\rm m}}

\def\r0{\mathcal{R}_0}

\newcommand{\medcap}{\mathbin{\scalebox{1.4}{$\cap$}}}
\newcommand{\medcup}{\mathbin{\scalebox{1.4}{$\cup$}}}

\def\gavg{\tau_{\rm m}}
\def\var{\xi}
\def\sinhc{{\rm sinhc}}
\def\cschc{{\rm cschc}}
\def\sinc{{\rm sinc}}

\def\tauA{\tau_1} 
\def\tauB{\tau_2} 

\def\ball{\mathcal{B}}
\def\sigmaeps{\mathcal{P}^{\epsilon}} 
\def\pieps{\mathcal{R}^{\epsilon}} 
\def\seps{S_{\epsilon}} 

\date{}
\title{Influence of Large Mean Delay on Distributed Delay Differential Equations Dynamics: Application to a Neural Mass Model}

\date{ }
\author{
Isam Al-Darabsah\footnote{Department of Mathematics and Statistics, Faculty of Science and Arts, Jordan University of Science and Technology, P.O. Box 3030, Irbid 22110, Jordan.}\textsuperscript{ ,}\footnote{Corresponding author.}\textsuperscript{ ,}{$^\|$}
\hspace{0.5cm}
Sue Ann Campbell\footnote{Department of Applied Mathematics and Centre for Theoretical Neuroscience, University of Waterloo, Waterloo, Canada.}\textsuperscript{ ,}{$^\|$}
 \hspace{0.5cm}
Bootan Rahman\footnote{Mathematics Unit, School of Science and Engineering, University of Kurdistan Hewler (UKH), Erbil 44001, Iraq.}\textsuperscript{ ,}\footnote{Center for Applied Mathematics and Bioinformatics (CAMB), Gulf University for Science and Technology, 32093, Hawally, Kuwait.}
\textsuperscript{ ,}\footnote{Email addresses: imaldarabsah@just.edu.jo (I. Al-Darabsah), sacampbell@uwaterloo.ca (S.A. Campbell), bootan.rahman\linebreak @ukh.edu.krd (B. Rahman).}
}

\begin{document}

\maketitle
 
\begin{abstract} 
Delay differential equations (DDEs) with large delays play a pivotal role in understanding stability and bifurcations in systems ranging from neural networks to laser dynamics. While prior work has extensively studied DDEs with discrete delays, the impact of distributed delays has been less explored. This paper investigates the spectrum of linear DDEs with a uniformly distributed delay kernel, with mean delay $ \gavg$ and 
a half-width of $\rho$. 
When  $\gavg\to\infty$, 
we carry out asymptotic analysis and show that the spectrum splits into (i) a strong critical spectrum referring to a finite set of isolated, pure imaginary eigenvalues that are unaffected by delay, (ii) an asymptotic strong spectrum consisting of a finite set of eigenvalues with limits that are determined by non-delayed terms in the model, and (iii) a pseudo-continuous spectrum consisting of infinitely many eigenvalues that limit on the imaginary axis, with real parts that scale linearly with the delay. 
Although this behavior is similar to the fixed delay case, the distributed delay introduces additional spectral features, including an infinite countable number of horizontal asymptotes in the pseudo-continuous spectrum at frequencies $\omega= k\pi/\rho$, where $k\in \mathbb{Z}\setminus\{0\}$. We validate our theoretical result through several examples and compare our findings with fixed-delay results from the literature.
Finally, we apply the results to study the stability and bifurcations of a Wilson-Cowan model with a delayed self-coupling, large mean delay, and homeostatic plasticity.
\end{abstract}


\section{Introduction}\label{section_introduction}

Time-delayed dynamical systems have been widely studied as a result of their relevance in various fields such as physics, biology, and engineering. The inclusion of delay in differential equations introduces complex behaviors and high dimensions, often leading to different instability phenomena. In particular, large mean self-coupling delays play a crucial role in shaping the long-term behavior of a single dynamical node, which can be observed in laser systems with delayed feedback, population dynamics, and neural networks \cite{wolfrum2006eckhaus, heil2001chaos, carr1993processing, giacomelli1998multiple}.
In scalar systems, only the presence of delays can induce non-trivial dynamics. For example, models of population regulation or circadian rhythms exhibit oscillations solely due to time-delayed feedback \cite{may2007theoretical, mackey1977oscillation}. In contrast, for systems with multiple variables, the delay terms interact with the intrinsic dynamics 
and can lead to multistability, complex bifurcation structures, and chaotic dynamics \cite{kuang1993delay, diekmann1995delay}. 
Moreover, analytical and numerical studies have also demonstrated that time delays can significantly alter stability regions and oscillatory thresholds, leading to phenomena such as Hopf bifurcations and complex attractors \cite{hale1993introduction, insperger2011semi,al2020}. This potential for rich dynamics makes it essential to include delays when modeling feedback mechanisms in engineering control, neural activity, and epidemiological models.

Delay differential equations (DDEs) characterized by large delays appear in many applications, including neural network models, laser dynamics, and population dynamics. For example, the authors in \cite{al2020} studied two identical weakly connected oscillators with time-delayed coupling. They showed that a large delay leads to the existence of in-phase and anti-phase solutions and the occurrence of various bifurcations. In \cite{marino2019excitable}, a semiconductor laser model with two large time delays was investigated. The authors assumed two hierarchically long-delayed feedback loops on an excitable semiconductor laser, which led to complex spatio-temporal dynamics for wave propagation patterns. The authors in \cite{otto2012delay} studied passively mode‑locked semiconductor lasers under long optical feedback delays. They incorporated a delay much larger than the internal pulse period and found that long feedback delays induce complex dynamics, including pulse timing jitter and bistability. The pattern formation in a semiconductor laser model with two long-delayed optical feedbacks was studied in \cite{yanchuk2014pattern}. In \cite{kashchenko2024logistic}, the author proposed a logistic equation with a large delay in a feedback loop and investigated how the dynamics change with increasing delay. The authors in \cite{hasan2023stability} studied an optoelectronic oscillator model with a large delay due to fiber-optic-based delay lines.

While DDEs with a small delay can often be treated as perturbations to ordinary differential equations (ODEs), a large delay is a qualitatively different regime. It is of significant theoretical and practical importance, as they represent the infinite-dimensional nature of the DDEs and show more dynamic complexity. For instance, when the time delay becomes very large and approaches infinity, it exhibits dynamics on multiple time scales, allows for the coexistence of aperiodic solutions, leads to an increase in the dimensionality of unstable manifolds and chaotic attractors, and shows hyperchaos with several positive Lyapunov exponents \cite{wolfrum2010complex,lichtner2011spectrum}.
Yanchuk et al.~have developed a theory for linear DDEs in the singular limit of a large delay, identifying two distinct components of the spectrum: the strong spectrum, governed by the non-delayed terms, and the asymptotically continuous spectrum, emerging from the delayed term and clustering near the imaginary axis as the delay increases \cite{yanchuk2006control, wang2024universal}. This framework has allowed a rigorous classification of delay-induced instabilities and the derivation of amplitude equations near bifurcation points \cite{wolfrum2006eckhaus, kozyreff2023multiple,yanchuk2017spatio}.
This spectral dichotomy has been used to analyze laser dynamics \cite{lichtner2011spectrum}, control schemes \cite{yanchuk2006control},  and neural synchronization \cite{yanchuk2005properties}. For example, Yanchuk et al. \cite{yanchuk2006control} showed that delayed feedback control of unstable steady states fails for large delay unless the system is near a bifurcation threshold, while Wolfrum and Yanchuk \cite{wolfrum2006eckhaus} linked the pseudo-continuous spectrum to the Eckhaus instability, a pattern-forming phenomenon in spatially extended systems.
For multiple delays, Ruschel and Yanchuk showed in \cite{Stefan2021} that the characteristic spectrum of linear DDEs with multiple large hierarchical delays splits into the strong and asymptotically continuous spectra. Moreover, they constructed spectral manifolds corresponding to each delay timescale and used them to determine stability and bifurcation behavior in the large‑delay limit.

The analysis of DDEs with discrete delays has received considerable attention and it becomes critical to understand how large delays influence the global and local behavior \cite{lichtner2011spectrum,wolfrum2010complex,yanchuk2006control, wang2024universal}. 
However, many real systems do not exhibit a single, fixed delay but rather a distribution of delays due to variability in transmission times or processing lags. For instance, synaptic transmission in neural systems, signal routing in networks, and ecological interactions often introduce distributed rather than discrete delays \cite{al2024distributed, carr1993processing}. Yet, the extension of the spectral framework to such systems, especially in the case of large mean delay, has not been systematically developed. 
In this work, we investigate the spectrum of linear DDEs with a uniformly distributed delay kernel, with mean delay $ \gavg$ and a half-width of $\rho$. 
Through asymptotic analysis when $\gavg\to\infty$ 
we show that the characteristic spectrum still splits into a strong critical spectrum, an asymptotic strong spectrum, and a pseudo-continuous spectrum, similar to the fixed delay case.
However, the distributed delay introduces additional spectral features, including an infinite countable number of horizontal asymptotes in the pseudo-continuous spectrum at frequencies $\omega= k\pi/\rho$, where $k\in \mathbb{Z}\setminus\{0\}$.  We also emphasize that wider distributions require larger mean delays for the asymptotic approximation to be accurate.

The paper is organized as follows. In Section \ref{section_preliminary}, we first present the necessary mathematical notation and key definitions. Then, we provide the main theorems for the spectrum of linear DDEs with a uniformly distributed delay kernel and derive the scaling properties of the spectrum for a large mean delay. 
Section \ref{section_examples} presents examples to illustrate our results in different cases and to compare our findings with existing fixed-delay results in the literature. 
In Section \ref{section_model}, we apply the results and study the dynamics of a single-node population model of excitatory and inhibitory neurons with a large mean delay.
Section \ref{sec_discussion}  provides a discussion of the results.
The Appendix contains the proofs of the properties of the asymptotic spectrum.

\section{Preliminaries and Main Results}\label{section_preliminary}

Solving nonlinear mathematical models is usually difficult; hence, analyzing their linearization at equilibrium points is a foundational technique to study the local stability, exhibit the oscillatory behaviors, and determine parameter thresholds at which qualitative changes, like bifurcations, occur. In this section, we study the spectrum  of the eigenvalues of a linear delay differential equation system derived from a nonlinear model with a uniform distribution kernel when the mean delay is large. Many nonlinear systems, such as neural field models or laser equations with feedback, undergo instabilities that are first detected through the eigenvalues of their linearization. Therefore, a precise characterization of the spectrum in the large-delay regime provides direct insight into when and how nonlinear oscillations, multistability, or chaotic regimes may arise. In particular, extending spectral theory from the classical fixed-delay case to the more realistic setting of distributed delays is essential for understanding real-world systems where signal transmission times are inherently variable.

In a series of papers, \cite{yanchuk2006control,lichtner2011spectrum}, Yanchuk and collaborators studied the eigenvalue spectrum of the following linear delay differential equation
\begin{equation}\label{model_linear_discrete}
    \frac{du}{dt}=A\,u(t)+ B\,u(t-\tau).
\end{equation}
In the limit $\tau\rightarrow\infty$, they showed that the spectrum, $\sigmaeps$, is comprised of two parts: the {\em strong spectrum}, $\sigmaeps_s$, and the {\em asymptotically continuous spectrum}, $\sigmaeps_{c}$. The eigenvalues of $\sigmaeps_s$ approach fixed values as $\tau_m\rightarrow\infty$ that can be determined from the matrix $A$. The eigenvalues of $\sigmaeps_c$ approach the imaginary axis as $\tau\rightarrow\infty$. They further showed each of these latter eigenvalues satisfies the asymptotic scaling
$\lambda= \gamma/\tau+i\omega +O\left(1/\tau^2\right),\ \text{for}\,\ \gamma, \omega\in\mathbb{R}$
and derived expressions for curves on which the scaled real part $\gamma$ lies. It is the purpose of this paper to extend these results to the case of a uniform distributed delay. In particular, we will show similar results hold with the discrete delay replaced by the mean delay of the distribution, if the variance is fixed and sufficiently small. 

Consider the linear delay differential equations  
of the form
\begin{equation}\label{model_linear_prel}
    \frac{du}{dt}=A\,u(t)+ B\,\int_{0}^{\infty}\,g(s)\, u(t-s)\, ds
\end{equation}
 according to a uniform distribution kernel 
\begin{equation}\label{unieq1}
	g(s)=\left\{\begin{array}{ccc}
		\displaystyle{\frac{1}{2\rho} },& \mbox{for} &  \gavg-\rho \leq s \leq \gavg+\rho, \\ \\
		\displaystyle{0},& & \mbox{otherwise},
	\end{array}\right. 
\end{equation}
which has the mean time delay $ \gavg$ and the variance  $\var= {\rho^2}/{3}$. The parameter $\rho$ controls the width and height of the distribution and must satisfy $0< \rho\le \tau_m$. 
Here $u \in \mathbb{R}^n$ and $A, B \in \mathbb{R}^{n \times n}$. 
  We fix $\rho\ll \gavg$ for sufficiently large $\gavg$.

Taking the Laplace transform of equation  \eqref{model_linear_prel} with an initial value of zero leads to the  characteristic equation
\begin{equation}\label{charact_eq_tau}
    \Delta(\lambda):=\det\left( \lambda I_n-A-B\,e^{-\lambda\gavg}\,\sinhc(\lambda\rho)\right)=0
\end{equation}
where $I_n$ is the $n\times n$ identity matrix  and
\[ \sinhc(x)=\left\{ \begin{array}{cc}
      \frac{\sinh (x)}{x},&\mbox{if } x\ne 0 \\
       1,& \mbox{if } x=0
  \end{array}\right.\]

We show in Appendix~\ref{sec:appendix} that the behaviour of the eigenvalues of \eqref{charact_eq_tau} as $\gavg\rightarrow\infty$ is as described by \cite{yanchuk2006control,lichtner2011spectrum}. Specifically, all eigenvalues lie in one of the following sets.
\begin{enumerate}
  \item[(i)] The strong critical spectrum $\mathcal{A}_0$, which has pure imaginary eigenvalues that do not depend on $\tau_m$.
    \item[(ii)] The strong spectrum $\mathcal{P}^\epsilon_s$,  which is a finite set of eigenvalues that approach some fixed value as  $\gavg\rightarrow\infty$. The set of limiting values of these eigenvalues is called the asymptotic strong spectrum, $\mathcal{A}_s$, which consists of two subsets: the asymptotic strong unstable spectrum, $\mathcal{A}_+$, and asymptotic strong stable spectrum, $\mathcal{A}_-$, that, respectively, contain eigenvalues with positive and negative real parts.   
 \item[(iii)] The pseudo-continuous spectrum $\mathcal{P}_c^{\epsilon}$, which has eigenvalues  with real parts that scale as $\epsilon=1/\gavg$, $0<\epsilon\ll 1$, that is, for $\lambda\in \mathcal{P}_c^{\epsilon}$, we have
\[\lambda=\epsilon \gamma+i\omega +O\left(\epsilon^2\right),\qquad \text{for}\,\gamma, \omega\in\mathbb{R},\gamma\ne 0\] 
Note that these eigenvalues all approach the imaginary axis as $\epsilon\rightarrow 0$ i.e., $\gavg\rightarrow\infty$.
\end{enumerate}
To describe these sets further, we will decompose the system \eqref{model_linear_prel}. 
Assume with a suitable change of coordinates, the matrix $B$ can be written as 
\begin{equation}\label{matrix_B}
B=\left[\begin{array}{ll}
0 & 0 \\
0 & \widebar{B}
\end{array}\right],    
\end{equation}
where $\widebar{B} \in \mathbb{R}^{d \times d}$ is an invertible matrix with  $1 \leq d \leq n$. Then, we can decompose the matrix $A$ as 
\begin{equation}\label{matrix_A}
A=\left[\begin{array}{ll}
A_1 & A_2 \\
A_3 & A_4
\end{array}\right],
\end{equation}
where   $A_1 \in \mathbb{R}^{(n-d) \times(n-d)}$ and $A_4 \in \mathbb{R}^{d \times d}$. 
Finally,  equation \eqref{charact_eq_tau} can be written as
\begin{equation}\label{model_linear_A_stable}
   \Delta(\lambda)=\det\left( \begin{bmatrix}
\lambda I_{n-d} - A_1 &  -A_2 \\
-A_3 & \lambda I_d - A_4 - \widebar{B}\,e^{-\lambda\gavg}\,\sinhc(\lambda\rho)
\end{bmatrix} \right) = 0,
\end{equation}
where $I_{n-d}$ and $I_d$ are identity matrices of sizes $n-d$ and $d$, respectively. Here, the matrix $A_1$ (resp. $A_4$) represents the subsystem that is not directly influenced by the delay  (resp. influenced by the delay). The off-diagonal matrices  $A_2$ and $A_3$ couple the two subsystems. Further, let $\sigma(M)$ denote the spectrum of a matrix $M$.
\subsection{Strong critical spectrum.}
Based on the definition above, we have 
\[ \mathcal{A}_0:=\left\{i\omega\in\mathbb{C} \,\mid\,  \Delta(i\omega)=0,\ \forall \tau_m\ge 0\right\}.\]
Simple calculations show that there are two ways that this can occur. First, if $\sinhc(i\omega\rho)=\sinc(\omega\rho)=0$ and $i\omega\in\sigma(A)$, then clearly $i\omega\in\mathcal{A}_0$.
Second, if the two subsystems are decoupled, i.e., either $A_2=0$ or $A_3=0$, then 
\[\Delta(\lambda)=\det(\lambda I_{n-d} - A_1)\det(\lambda I_d - A_4 - \widebar{B}\,e^{-\lambda\gavg}\,\sinhc(\lambda\rho))=0.\]
In this case, $i\omega\in\sigma(A_1)$ implies $i\omega\in\mathcal{A}_0$.
In summary, we have
\[\mathcal{A}_0:=\left\{i\omega\in\mathbb{C} \,\mid\,i\omega\in\sigma(A) \mbox{ and } i\omega\in \sigma(A_1) \mbox{ or } \omega =k\pi/\rho,\ k\in\mathbb{Z}\setminus \{0\}\right\}.\]

Note that the first situation can only occur in the system with uniform distributed delay, while the second is identical for the system with fixed and distributed delay. Thus   $\mathcal{A}_0$ for a system with a fixed delay is a subset of that for the same system with a uniform distributed delay.

\subsection{Asymptotic strong  spectrum.}\label{asymptotic_spectrum}
 Theorem~\ref{thm_append_1} proves that asymptotic strong spectra are given as follows.
 \begin{itemize}
 \item[] Asymptotic strong unstable spectrum:
 \begin{equation}
\mathcal{A}_{+}:=\{\lambda\in\sigma(A) \, \mid   \, \Re(\lambda)>0\}.
\label{A+def}
\end{equation}
 \item[] Asymptotic strong stable spectrum:
 \begin{equation}
\mathcal{A}_{-}:=\{\lambda\in\sigma(A_1) \, \mid   \, \Re(\lambda)<0\}.
\label{A_def}
\end{equation}
  \item[] Asymptotic strong  spectrum:
   \[
\mathcal{A}_s=\mathcal{A}_{+} \cup \mathcal{A}_{-}.
\]
\end{itemize}
These can be understood heuristically as follows.

Let $\lambda=\lambda_0+\lambda_{\gavg}$ be an eigenvalue of \eqref{charact_eq_tau}  where $\lambda_{\gavg}\rightarrow 0$ and ${\gavg}\lambda_{\gavg}\rightarrow K<\infty$ as ${\gavg}\rightarrow\infty$. 
If $\Re(\lambda_0)>0$, then it follows  from equation \eqref{model_linear_A_stable} that $\det (\lambda_0 I-A)=0$ as $\gavg\to \infty$ because $e^{-\lambda\tau_m}\rightarrow 0$. 
If $\Re(\lambda_0)<0$, note that \eqref{model_linear_A_stable} is equivalent to
\begin{equation}\label{model_linear_A_ustable_1}
\det\left( \begin{bmatrix}
\lambda I_{n-d} - A_1 &  -A_2 \\
-A_3\,e^{\lambda\gavg} & (\lambda I_d - A_4)\,e^{\lambda\gavg} - \widebar{B}\,\sinhc(\lambda\rho)
\end{bmatrix} \right) = 0.
\end{equation}
Thus  as  $\gavg\to \infty$, we have
\begin{equation}\label{model_linear_A_ustable}
  \det\left( \begin{bmatrix}
\lambda_0 I_{n-d} - A_1 &  -A_2 \\
0 &  - \widebar{B}\,\sinhc(\lambda_0\rho)
\end{bmatrix} \right) = 0
\end{equation}
because $e^{\lambda\tau_m}\rightarrow 0$, which is satisfied if and only if $\det\left(\lambda_0 I_{n-d} - A_1\right)=0$.

Note that the asymptotic strong spectra are the same for a fixed delay as a for a uniform distribution of delays. This would be true for any distribution of delays.

\subsection{Pseudo-continuous spectrum.}\label{pseudo_spectrum}

To study the  pseudo-continuous spectrum $\mathcal{P}_c^{\epsilon}$, we find an  asymptotic  continuous spectrum $\mathcal{A}_c$ such that $\Re(\overline{\lambda})\approx\epsilon\,\Re(\tilde{\lambda})$ and $\Im(\overline{\lambda})\approx \Im(\tilde{\lambda})$ for any $\overline{\lambda}\in \mathcal{P}_c^{\epsilon}$ and $\tilde{\lambda}\in \mathcal{A}_c$. To this end, let 
\[
\lambda=\frac{\gamma}{\gavg}+i\omega=\epsilon\gamma+i\omega,\qquad \gamma, \omega\in\mathbb{R},\gamma\ne 0.
\] 
It then follows from equation \eqref{charact_eq_tau} that 
\begin{equation} \label{eq_111}
  \det\left([\epsilon\gamma+i \omega] I_n-A-B\,e^{-\gamma}\,e^{-i\gavg\omega}\,\sinhc( \rho(\epsilon\gamma+i \omega)\right)=0.
\end{equation}
 Note that 
\begin{equation}\label{eq_O2_sinhc}
    \sinhc(\rho(\epsilon\gamma+i \omega))=
\frac{\sin (\rho\omega)}{\rho\omega}+
i \epsilon \rho\gamma f_1(\omega)
+\gamma ^2 \epsilon ^2 \rho ^2 f_2(\omega)
+O\left(\epsilon ^3\right),
\end{equation}
where
\[
f_1(\omega)=\frac{\sin(\rho\omega)-\rho\omega\cos(\rho\omega)}{  \rho^2\omega^2} \qquad \text{and} \qquad 
f_2(\omega)=\frac{\rho ^2 \omega ^2 \sin (\rho  \omega )-2 \sin (\rho  \omega )+2 \rho  \omega  \cos (\rho  \omega )}{2 \rho^3  \omega ^3}.
\]
For any fixed $\rho>0$ and $\omega\in\mathbb{R}$, we have $|f_1(\omega)|\le 0.4361$ and $-0.1243\le f_2(\omega)< 1/6$. 
Let $\varphi=\gavg\omega$. Then, for sufficiently small $0<\epsilon\ll1$, to leading order, we have
 \begin{equation}\label{eq_zzzzq}
  \det\left(i \omega I_n-A-B\,e^{-(\gamma+i\varphi)}\,\sinc( \rho\omega)\right)=0
\end{equation}
where
 \[ \sinc(x)=\left\{ \begin{array}{cc}
      \frac{\sin(x)}{x},&\mbox{if } x\ne 0 \\
       1,& \mbox{if } x=0.
  \end{array}\right.\]

  \begin{remark}\label{remark_rho}
  Recall that the model is well defined only if $\rho<\tau_m$. Thus  given fixed $\rho>0$, for the expansion above to be valid we require $\epsilon<1/\rho$ and $\epsilon$ sufficiently small such that $0.4361\rho =O(1)$ with respect to $\epsilon$, which will be satisfied if $\epsilon\ll 3/\rho$ 
\end{remark}

Note that for large values of $\gavg$, the term $e^{-i\varphi}$ oscillates rapidly with respect to $\omega$. However, for any fixed $\epsilon$ sufficiently small,  $\gamma$ will satisfy \eqref{eq_zzzzq} to $O(\epsilon)$ for some value of $\varphi$. Furthermore, the Euler formula can eliminate $\varphi$  and express $\gamma$ as a function of $\omega$. Considering this, we treat $\varphi$ as an auxiliary phase parameter and seek the solution to equation \eqref{eq_zzzzq} as a curve in the $(\gamma, \omega)$-plane.
Thus  we define the following polynomial in $x$
\[
p_\omega(x)= \det\left(i \omega I_n-A-x\,\sinc( \rho\omega)\,B\right), 
\]
and the set
\begin{equation}\label{set_1}
    \mathcal{W}=\left\{\omega\in\mathbb{R}\mid i\omega\notin\sigma(A_1)\text{ and } \rho\omega\neq k\pi\text{ for } k\in \mathbb{Z}\setminus\{0\}\right\}.
\end{equation}

\begin{theorem}\label{Thm_App_1}
  The degree of $p_\omega$ equals $\operatorname{rank} (B)=d$ if and only if $\omega\in\mathcal{W}$.
\end{theorem}
\begin{proof}
 It follows from the assumed form of $B$ in \eqref{matrix_B} that the degree of $p_\omega$ is at most $d$. Assume,  without loss of generality, that $\widebar{B}$ is in Jordan canonical form containing the eigenvalues $\widebar{\lambda}_j$, $j=1,\ldots,d$, on its diagonal.
 By \cite[Lemma 7]{lichtner2011spectrum},  the  leading order coefficient for monomial $x^d$ in $p_\omega(x)$ is
\[
c_d:=\sinc^d(\rho\omega)\,\det\left(i \omega I_{n-d}-A_1\right) \prod_{j=1}^d \widebar{\lambda}_j.
\]
Since $\widebar{B}$ is an invertible matrix, $\widebar{\lambda}_j\neq0$ for all $j=1,\ldots,d$. Hence
\[c_d\neq 0 \Leftrightarrow  i \omega \notin \sigma\left(A_1\right)
\text{ and }
\rho\omega\neq k\pi\text{ for } k\in \mathbb{Z}\setminus\{0\}.
\]
\end{proof}
By Theorems \ref{Thm_App_1} and \ref{theorem_rescaled_pseudo_continuous}, there exist $d$ spectral curves $\gamma_{\ell}(\omega) $ that  satisfy
\[p_\omega\left(e^{-(\gamma_{\ell}(\omega)+i\varphi)}\right)=0,\qquad \text{for some} \,  \varphi\in\mathbb{R}.\]
Consequently, we define the asymptotic continuous spectrum  as 
\begin{equation}\label{A_c}
    \mathcal{A}_c:=\bigcup_{\ell
 = 1}^d\left\{\gamma_{\ell}(\omega)+i \omega \in \mathbb{C} \mid  \,\exists \varphi \in \mathbb{R}:p_\omega\left(e^{-(\gamma_{\ell}(\omega)+i\varphi)}\right)=0,  \, 
\gamma_{\ell}(\omega) \notin \left\{ - \infty ,\infty \right\}\right\}
\end{equation}
For the singularities $\left\{ - \infty ,\infty \right\}$ in \eqref{A_c}, it follows from Theorem \ref{theorem_rescaled_pseudo_continuous} that  for $\ell\in\{1,\ldots,d\}$, we have
\begin{align*}
 ~& \gamma_{\ell}(\omega)=\infty 
\Leftrightarrow
i\omega\in \sigma(A) \\
~&\gamma_{\ell}(\omega)=-\infty \Leftrightarrow
i\omega\in \sigma(A_1).
\end{align*}
The previous definition applies if $\mathcal{A}_0=\emptyset$. If not, then the definition of the asymptotic continuous spectrum becomes 
\begin{equation}\label{eqq_A0_new}
 \mathcal{A}_c:=\mathrm{closure}\left\{\gamma+i \omega \in \mathbb{C} \mid i \omega \notin \mathcal{A}_0 \text { and } \exists \varphi \in \mathbb{R}: p_\omega\left(e^{-\gamma} e^{-i \varphi}\right)=0\right\}.  
\end{equation}
Note that we can use the notation introduced here to give an alternate definition of $\mathcal{A}_0$ 
\[
\mathcal{A}_0:=\left\{i \omega \in \mathbb{C} \,\mid\,  i \omega \in \sigma(A) \,\text{ and }\,  p_\omega(x)=0 \text { for all } x\in \mathbb{C}\right\}.
\]

\subsection{Stability results.}
Based on the results in the Appendix \ref{sec:appendix}, the following results characterize the stability of \eqref{model_linear_prel} in terms of the spectral curves.
 \begin{theorem}
\begin{enumerate}[series=MyList,label= ({\roman*})]
    \item  If the spectral curves are in the negative half-plane, i.e., $\gamma_l(\omega)<0$ for all $\omega \in \mathbb{R}, 1 \leq l \leq d$, and $\mathcal{A}_{+}=\emptyset$, then there exists $0<\epsilon_0\le 1/\rho$ such that for $0<\epsilon<\epsilon_0$ the equation \eqref{model_linear_prel} is exponentially asymptotically stable.

   \item  If some spectral curve admits positive values, i.e., $\gamma_l(\omega)>0$ for some $\omega \in \mathbb{R}$ and $l \in\{1, \ldots, d\}$, or $\mathcal{A}_{+} \neq \emptyset$, then there exists $0<\epsilon_0\le 1/\rho$ such that for $0<\epsilon<\epsilon_0$ the the equation \eqref{model_linear_prel} is exponentially asymptotically unstable.
\end{enumerate}
\end{theorem}

  \begin{theorem}
  There exists $0<\epsilon_0\le 1/\rho$ such that for $0<\epsilon<\epsilon_0$ the following assertion is true: Assume that $i \omega$ is the eigenvalue in $\sigmaeps$ with largest real part, and that the polynomial $p_\omega(x)$ is not identically zero. Then,  the eigenvalue $ i \omega$ belongs to the pseudo-continuous spectrum.
  \end{theorem}

\section{Examples}\label{section_examples}
This section presents three modified examples from \cite{lichtner2011spectrum} to compare our results with those presented there. The fixed delay case $\tau$ corresponding to the Dirac-delta function $\delta(s-\tau)$ in \eqref{model_linear_prel} is discussed in \cite{lichtner2011spectrum}. Recall that as $\rho$ approaches zero in \eqref{unieq1},  the uniform distribution kernel becomes the Dirac-delta function. 
One of the key differences from  \cite{lichtner2011spectrum} that we noticed in the analysis and computations is that in all of the examples, the spectral curve approaches infinity at an infinite countable number of point $\omega= k\pi/\rho$, where $k\in \mathbb{Z}\setminus\{0\}$.

\begin{example}\label{example_1}
    Consider the scalar equation 
    \begin{equation}\label{example1_eq}
    \frac{du}{dt}=\alpha\,u(t)+\frac{\beta}{2\rho}\,\int_{\tau_m-\rho}^{\tau_m+\rho} u(t-s)\, ds
\end{equation}
where $\alpha,\,\beta\in\mathbb{R}$. We assume $\beta\ne 0$, so that the equation is a DDE.
\end{example}
It is clear that $A=A_4=\alpha$ and $B=\widebar{B}=\beta$. Then,  $\mathcal{A}_{-}=\mathcal{A}_{0}=\emptyset$ and 
\[
\mathcal{A}_{s}=\mathcal{A}_{+}=\{\lambda-\alpha=0 \, \mid   \, \Re(\lambda)>0\}= \left\{ {\begin{array}{*{20}{c}}
{  \alpha  ,}&{\alpha > 0,}\\
\emptyset,&{\alpha < 0.}
\end{array}} \right.\]
Thus  we immediately see that, for all $\rho\ge 0$ and all $\gavg$ sufficiently large, system~\eqref{example1_eq} is unstable when $\alpha>0$.

The characteristic equation corresponding to \eqref{example1_eq} is 
\begin{equation}\label{eq_vvv}
    \Delta_1(\lambda)=\lambda -\alpha-\beta\, e^{-\lambda\tau_m}\,\sinhc( \rho\lambda)=0.
\end{equation}
Consequently, we define 
\[
p_\omega(x)=i \omega -\alpha-x\,\sinc( \rho\omega)\,\beta.
\]
When $\beta\neq0$, we have $\text{rank}(B)=1$. Hence,  there exists one spectral curve that  satisfies
\begin{equation}\label{eq_1_ex_1}
    p_\omega\left(e^{-(\gamma (\omega)+i\varphi)}\right)=0,\qquad \text{for some} \,  \varphi\in\mathbb{R}.
\end{equation}
Separating the real and imaginary parts gives
\begin{neweq_non}\label{eq:ReIm}
\alpha\,+\,\beta\, e^{-\gamma(\omega)}\cos(\varphi)\,\sinc( \rho\omega)=0,\\
\omega\,+\,\beta\, e^{-\gamma(\omega)}\sin(\varphi)\,\sinc( \rho\omega)=0.
\end{neweq_non}
Using the Pythagorean trigonometric identity, we eliminate $\varphi$ and get the spectral  curve
\[  \gamma(\omega)=-\frac{1}{2}\ln\left(\frac{   \alpha ^2+\omega ^2 }{\beta ^2\, \sinc ^2(\rho  \omega )}\right).  \]
Consequently, we have
\[
\mathcal{A}_c=\left\{\gamma(\omega)+i \omega \,\mid\, \omega \in \mathbb{R},\, \omega \neq0 \text{ or } \alpha \neq 0\right\}.
\]
Figures \ref{Fig:Example_1_1} and \ref{Fig:Example_1_2} show consistency with the theoretical predictions.
We use the command \textsf{Solve} in \textit{Wolfram Mathematica} to find the eigenvalues of the characteristic equation \eqref{eq_vvv} in a circle of radius twenty units centered at the origin. 
The numerically computed eigenvalues follow the pseudo-continuous spectrum curve and become more closely aligned with the curve as $\gavg$ increases with fixed $\rho$ and as $\rho$ decreases with fixed large $\gavg$. Moreover, the horizontal asymptotes appear at the expected $\omega = k\pi/\rho$,  $k\in \mathbb{Z}\setminus\{0\}$,  which become denser and closer to each other as $\rho$  increases.  Moreover, when $\alpha=0$, we have that  $\gamma(\omega)$ has a horizontal asymptote at $\omega=0$ since $\gamma(0)\to \infty$, see Figure \ref{Fig:Example_1_2A}.

Note that at $\omega=0$, the numerator $\alpha ^2+\omega ^2$ has the absolute minimum $\frac{1}{2}\ln\left(\alpha ^2 \right)$ and the denominator  $\beta ^2\, \sinc ^2(\rho  \omega )$  has the absolute maximum $\frac{1}{2}\ln\left(\beta ^2 \right)$. Hence, the absolute maximum value of $ \gamma(\omega)$ is $\gamma^*=\frac{1}{2}\ln(\beta^2/\alpha^2)=\ln|\beta/\alpha|$ occurs at $\omega=0$. Note that the value of $\gamma^*$ is independent of  $\rho$. Consequently,  when $\alpha<0$, for all $\rho\ge 0$ and $\gavg$ sufficiently large, system  \eqref{example1_eq} is
\begin{equation}\label{stbility_example_1}
    \text{stable if }\, \alpha<-|\beta|<0
   \qquad\text{and}\qquad
   \text{unstable if }\, -|\beta|<\alpha<0,
\end{equation}
see Figures \ref{Fig:Example_1_3}A and \ref{Fig:Example_1_3}C. Note that the eigenvalue spectrum touches the imaginary axis when $\alpha=-|\beta|$. This stability condition  is shown in Figure \ref{Fig:Example_1_3}B. Further, the fact that $\gamma^*$ is independent of  $\rho$ can be seen in numerical computations of the spectra presented in Figure \ref{Fig:Example_1_1}, where increases in $\rho$ demonstrate no effect on the maximum real part of the continuous spectrum.

As  $\rho\to 0$, it follows from $g(s)$ in \eqref{unieq1} that 
\[\lim_{\rho \rightarrow 0}\int_{0}^{\infty}\,g(s)\, u(t-s)\, ds= \lim_{\rho \rightarrow 0}\int_{-\infty}^{\infty}\,g(s)\, u(t-s)\, ds=\int_{-\infty}^{\infty}\,\delta(s-\tau)\, u(t-s)\, ds=u(t-\gavg).\]
Consequently, system \eqref{example1_eq} transitions to a fixed delay equation, which can be expressed as
\[ \frac{du}{dt}=\alpha\,u(t)+\beta\,u(t-\gavg).\]
Compared to the fixed delay case, the stability condition  \eqref{stbility_example_1} remains unchanged, as previously discussed in \cite[Example 1]{lichtner2011spectrum}. As noted in Remarks~\ref{remark_rho} and \ref{rem_rhodep} when $\rho$ is larger, larger values of $\tau_m$ are need to see good correspondence between the true spectra and the asymptotic spectra.

\begin{figure} [!htb]
    \centering
    \includegraphics[width=1\linewidth]{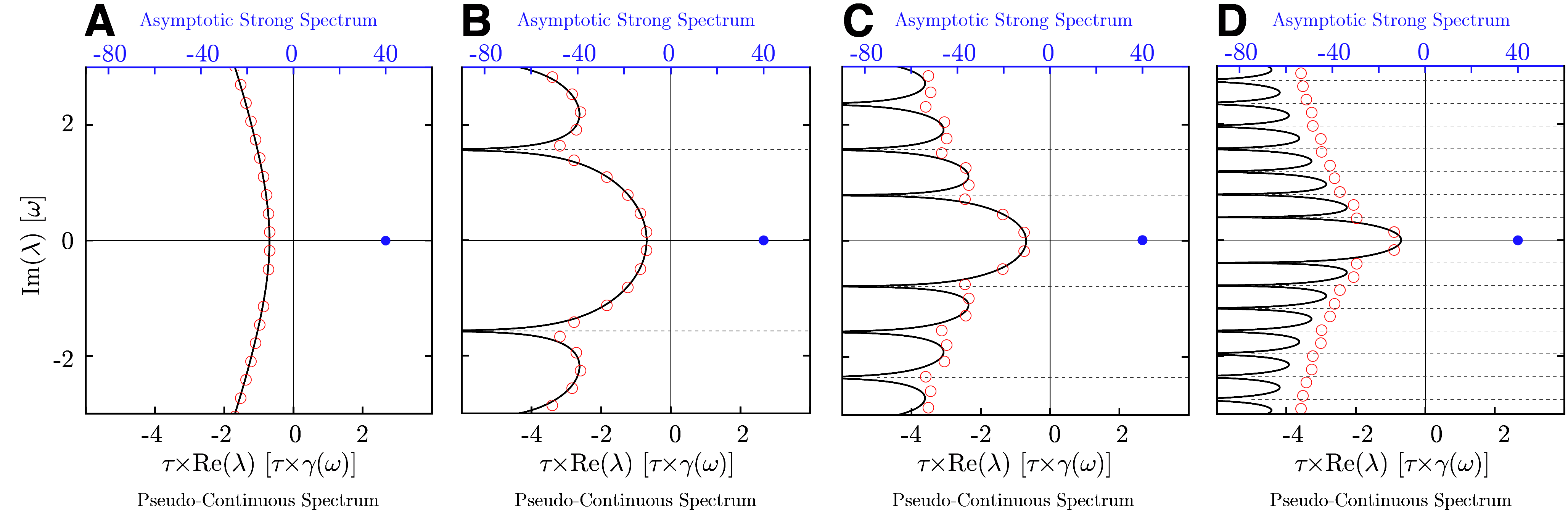}
\caption{The effect of increasing $\rho$ with a fixed large $\gavg$  on the asymptotic calculations in Example \ref{example_1}. Spectra of Example \ref{example_1} for large mean delay with $\gavg=20$, $\alpha=2$, and $\beta=1$:  {\rm (A)} $\rho=0.5$      {\rm (B)} $\rho=2$   {\rm (C)} $\rho=4$      {\rm (D)} $\rho=8$. 
Black curve is the pseudo-continuous spectrum, blue dots are the asymptotic strong spectrum, red circles are numerically computed eigenvalues, and gray dashed lines are horizontal asymptotes at $\omega= k\pi/\rho$, $k\in \mathbb{Z}\setminus\{0\}$.}
\label{Fig:Example_1_1}
\end{figure}

 \begin{figure} [!htb]
    \centering
    \includegraphics[width=1\linewidth]{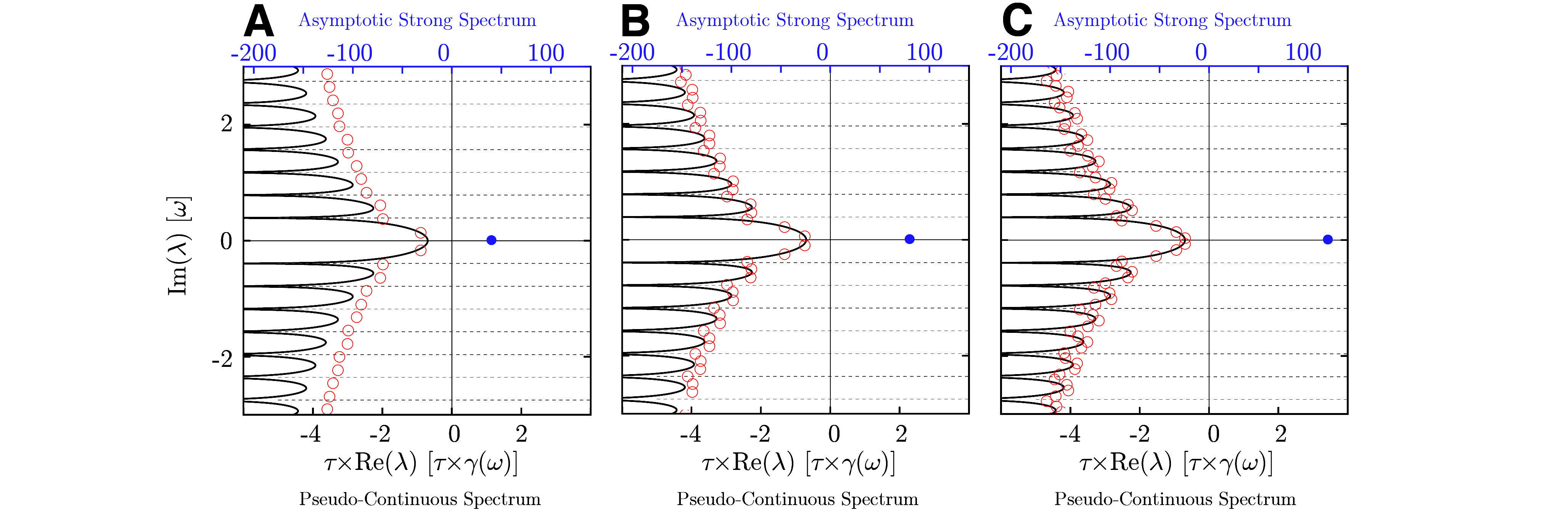}
    \caption{The effect of increasing $\gavg$ with a fixed  $\rho$  on the asymptotic calculations in Example \ref{example_1}. Spectra of Example \ref{example_1} for large mean delay with $\rho=8$, $\alpha=2$, and $\beta=1$: {\rm (A)} $\gavg=20$   {\rm (B)} $\gavg=40$      {\rm (C)} $\gavg=60$.  
Black curve is the pseudo-continuous spectrum, blue dots are the asymptotic strong spectrum, red circles are numerically computed eigenvalues, and gray dashed lines are horizontal asymptotes at $\omega= k\pi/\rho$,  $k\in \mathbb{Z}\setminus\{0\}$.}
\label{Fig:Example_1_2}
\end{figure}

\begin{figure} [!htb]
    \centering
    \includegraphics[width=1\linewidth]{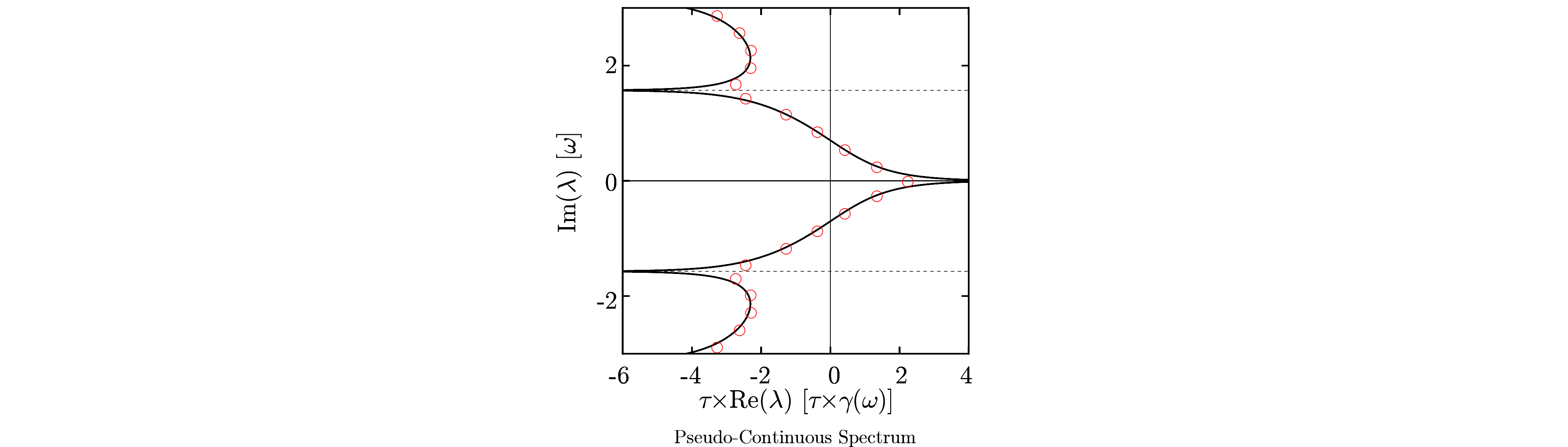}
\caption{Spectra of Example \ref{example_1} when $\alpha=0$.  Parameter values are $\gavg=20$, $\rho=2$, and $\beta=2$.
Black curve is the pseudo-continuous spectrum, red circles are numerically computed eigenvalues, and gray dashed lines are horizontal asymptotes at $\omega= k\pi/\rho$, $k\in \mathbb{Z}\setminus\{0\}$.}
\label{Fig:Example_1_2A}
\end{figure}

 \begin{figure} [!htb]
    \centering
    \includegraphics[width=1\linewidth]{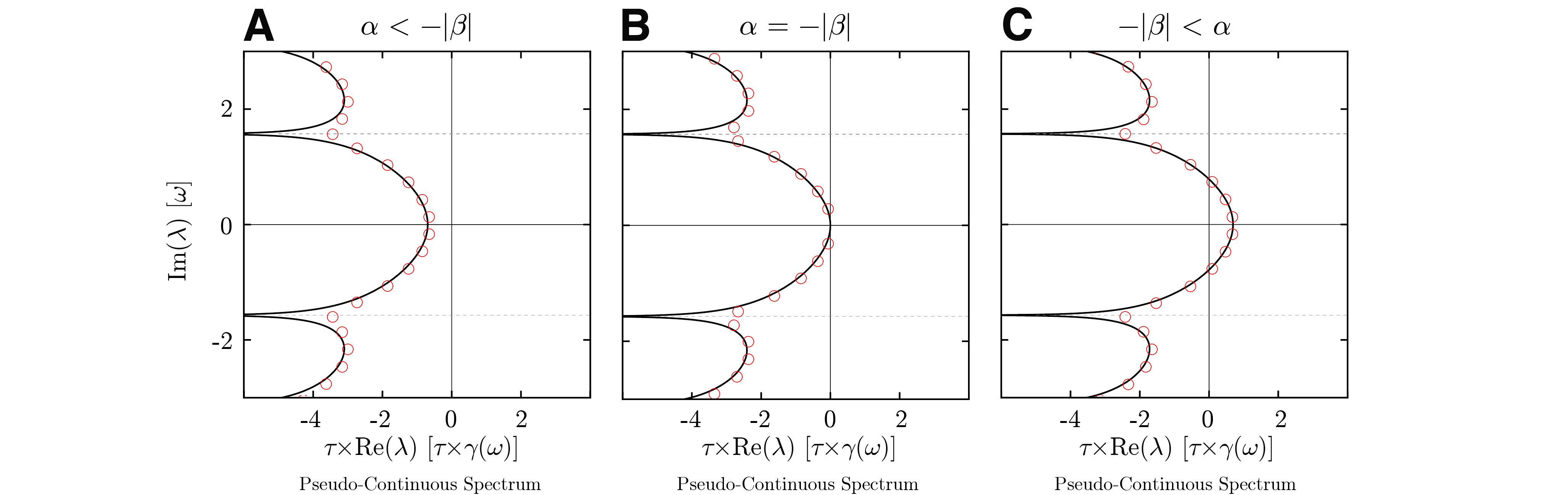}
\caption{Stability condition  \eqref{stbility_example_1}  in Example \ref{example_1}. Spectra of Example \ref{example_1} for large mean delay $\gavg=20$, $\rho=2$, $\alpha=-1$:  {\rm (A)} $\beta=-0.5$   {\rm (B)} $\beta=-1$      {\rm (C)} $\beta=2$.
Black curve is the pseudo-continuous spectrum, red circles are numerically computed eigenvalues, and gray dashed lines are horizontal asymptotes at $\omega= k\pi/\rho$, $k\in \mathbb{Z}\setminus\{0\}$.}
\label{Fig:Example_1_3}
\end{figure}

  
\begin{example}\label{example_2}
    Consider the system 
    \begin{equation}\label{example2_eq}
    \frac{du}{dt}=\left[\begin{array}{cc}
\alpha & 1 / 2 \\
2 & 0
\end{array}\right]\,u(t)+ \frac{1}{2\rho} \,\int_{\tau_m-\rho}^{\tau_m+\rho} 
\left[\begin{array}{ll}
0 & 0 \\
0 & 1
\end{array}\right]\,u(t-s) \, ds
\end{equation}
where $u(t)=[u_1(t),u_2(t)]^T$ and $\alpha \in\mathbb{R}$.
\end{example}
In this example, we have 
\[A=
\left[\begin{array}{c|c}
A_1 & A_2 \\\hline
A_3 & A_4
\end{array}\right]=
\left[\begin{array}{c|c}
\alpha & 1 / 2 \\\hline
2 & 0
\end{array}\right]
\qquad\text{and}\qquad
B=
\left[\begin{array}{c|c}
0 & 0 \\\hline
0 & \widebar{B}
\end{array}\right]=
\left[\begin{array}{c|c}
0 & 0 \\\hline
0 & 1
\end{array}\right].
\]
Since $\sigma(A)$ only contains real, non-zero eigenvalues $\mathcal{A}_0=\emptyset$.
Further, we have
\begin{align*}
    \mathcal{A}_{-}&=\{\lambda-A_1=0 \, \mid   \, \Re(\lambda)<0\}= \left\{ {\begin{array}{*{20}{c}}
{ \emptyset   ,}&{\alpha \ge 0,}\\
 \alpha,&{\alpha < 0.}
\end{array}} \right.\\
       \mathcal{A}_{+}&=\{\det(\lambda\,I_2-A)=0 \, \mid   \, \Re(\lambda)>0\}= \left\{ {\begin{array}{*{20}{c}}
{ \frac{\alpha}{2}+\sqrt{1+\frac{\alpha^2}{4}}  ,}&{\alpha > 0,}\\
\emptyset,&{\alpha \le 0.}
\end{array}} \right.
\end{align*}
Hence, we obtain 
\[  \mathcal{A}_{s}= \left\{ {\begin{array}{*{20}{c}}
{ \frac{\alpha}{2}+\sqrt{1+\frac{\alpha^2}{4}}  ,}&{\alpha > 0,}\\
\alpha,&{\alpha< 0.}
\end{array}} \right.\]
Thus  as in the previous example, for all $\rho\ge 0$ and all $\gavg$ sufficiently large, system~\eqref{example2_eq} is unstable when $\alpha>0$.

The characteristic equation corresponding to \eqref{example2_eq} is 
\[\Delta_2(\lambda)=\lambda^2-\alpha\lambda-1 -(\lambda-\alpha) e^{-\lambda\tau_m}\,\sinhc( \rho\lambda)=0
\]
Hence, we define 
\[
p_\omega(x)=-\omega^2-\alpha\,i\,\omega-1 -(i\,\omega-\alpha)\,x\,\sinc( \rho\omega).
\]
Since  $\text{rank}(B)=1$,  there exists one  spectral curve 
$  \gamma(\omega)$ 
that  satisfies
$p_\omega\left(e^{-(\gamma (\omega)+i\varphi)}\right)=0$ for some $\varphi\in\mathbb{R}$. Separating the real and imaginary parts, we obtain 
\begin{neweq}\label{eq:ReIm_2}
\omega^2+1&=(\omega\sin(\varphi)-\alpha\cos(\varphi))\,\sinc( \rho\omega)\, e^{-\gamma(\omega)}\\
\alpha\omega &=(\omega\cos(\varphi)+\alpha\sin(\varphi))\,\sinc( \rho\omega)\, e^{-\gamma(\omega)}
\end{neweq}
Then,  squaring and adding equation \eqref{eq:ReIm_2} lead to
\[
\gamma(\omega)=-\frac{1}{2} \ln \left(\frac{\left(1+\omega^2\right)^2+\alpha^2 \omega^2}{\left(\alpha^2+\omega^2\right)\,\sinc^2( \rho\omega)}\right)
\]
Consequently, we have
\[
\mathcal{A}_c=\left\{\gamma(\omega)+i \omega \,\mid\, \omega \in \mathbb{R},\, \omega \neq0 \text{ or } \alpha \neq 0\right\}.
\]
Figure \ref{Fig:Example_2_1} shows the same behavior as in Figures \ref{Fig:Example_1_1}  and  \ref{Fig:Example_1_2}: the numerically computed eigenvalues remain consistent with the pseudo-continuous spectrum predicted by theory. In particular, for fixed $\gavg$ when $\rho$ becomes smaller, the numerical results align more closely with the pseudo-continuous curve.

Simple calculations show that $\gamma(\omega)$ always has a critical point at $\omega=0$, which is a global maximum for all $\rho\ge 0$ if $|\alpha|>\alpha^*:=\sqrt{\sqrt{2}-1}$. Further, if $0<|\alpha|\le \alpha^*$ then $\gamma(0)=\ln\left|\alpha\right |$ is a global maximum if $\rho>\rho^*:=\sqrt{\frac{3}{\alpha^2}-6-3\alpha^2}$ and a local minimum if $\rho<\rho^*.$
Since $\gamma(0)$ is independent of $\rho$ we conclude that the system will be asymptotically stable for all $\rho\ge 0$ and sufficiently large $\gavg$ if $-1<\alpha<-\alpha^*$, and for all $\rho>\rho^*$ and all sufficiently large $\gavg$ if $-\alpha^*<\alpha<0$. Figure~\ref{Fig:Example_2_1} illustrates the case $\alpha>\alpha^*$. It can be seen that the maximum real part of the pseudo-continuous spectrum occurs at $\omega=0$ and does not vary with $\rho$.

When $\alpha=0$, note that $\gamma(\omega)$ has a horizontal asymptote at $\omega=0$ since $\gamma(0)\to -\infty$, see Figure \ref{Fig:Example_2_2}B. This also occurs in the fixed delay case \cite[Example 3]{lichtner2011spectrum}:
\[
    \frac{du}{dt}=\left[\begin{array}{cc}
\alpha & 1 / 2 \\
2 & 0
\end{array}\right]\,u(t)+  
\left[\begin{array}{ll}
0 & 0 \\
0 & 1
\end{array}\right]\,u(t-\gavg).  
\]
Note that the curves of the asymptotic pseudo-continuous spectrum for the discrete delay case are given by 
\[
\gamma_{ fixed}(\omega)=-\frac{1}{2} \ln \left(\frac{\left(1+\omega^2\right)^2+\alpha^2 \omega^2}{\left(\alpha^2+\omega^2\right)}\right)\ge \gamma(\omega).
\]
Thus  we may expect the model with distributed delay to be more stable than that with fixed delay.
Figure \ref{Fig:Example_2_2} shows a comparison with \cite[Figure 3.2]{lichtner2011spectrum}. The parameters in  Figures \ref{Fig:Example_2_2}A and  \ref{Fig:Example_2_2}B correspond to $|\alpha|<\alpha^*$ and $\rho<\rho^*$.
We notice that the pseudo-continuous spectrum curve with a distributed delay has horizontal asymptotes at an infinitely countable number of points $\omega= k\pi/\rho$, where $k\in \mathbb{Z}\setminus\{0\}$, and an additional horizontal asymptote at $\omega=0$ when $\alpha=0$. On the other hand, only one horizontal asymptote exists at $\omega=0$ when $\alpha=0$ in the case of fixed delay.

\begin{figure} [!htb]
    \centering
    \includegraphics[width=1\linewidth]{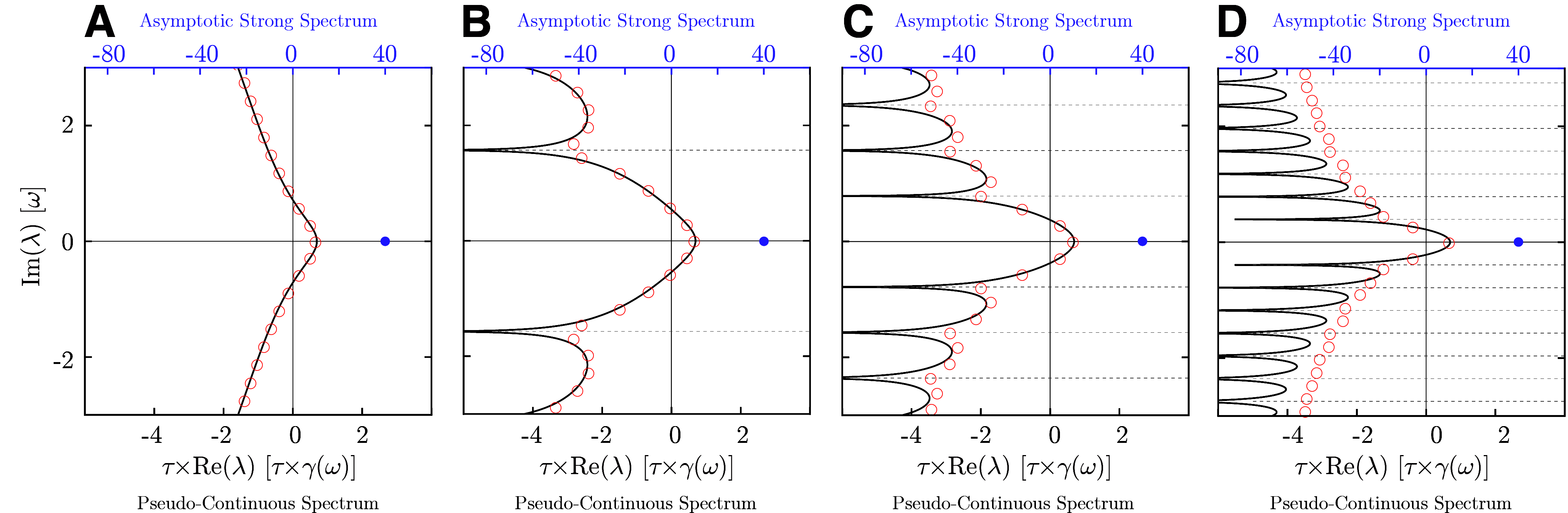}
    \caption{The effect of increasing $\rho$ with a fixed large mean delay $\gavg$ on the asymptotic calculations  in Example \ref{example_2}. Spectra of Example \ref{example_2} for large mean delay with $\gavg=20$, and $\alpha=2$:  {\rm (A)} $\rho=0.5$      {\rm (B)} $\rho=2$   {\rm (C)} $\rho=4$      {\rm (D)} $\rho=8$. 
Black curve is the pseudo-continuous spectrum, blue dots are the asymptotic strong spectrum, red circles are numerically computed eigenvalues, and gray dashed lines are horizontal asymptotes at $\omega= k\pi/\rho$, where $k\in \mathbb{Z}\setminus\{0\}$.}
\label{Fig:Example_2_1}
\end{figure}

\begin{figure} [!htb]
    \centering
    \includegraphics[width=1\linewidth]{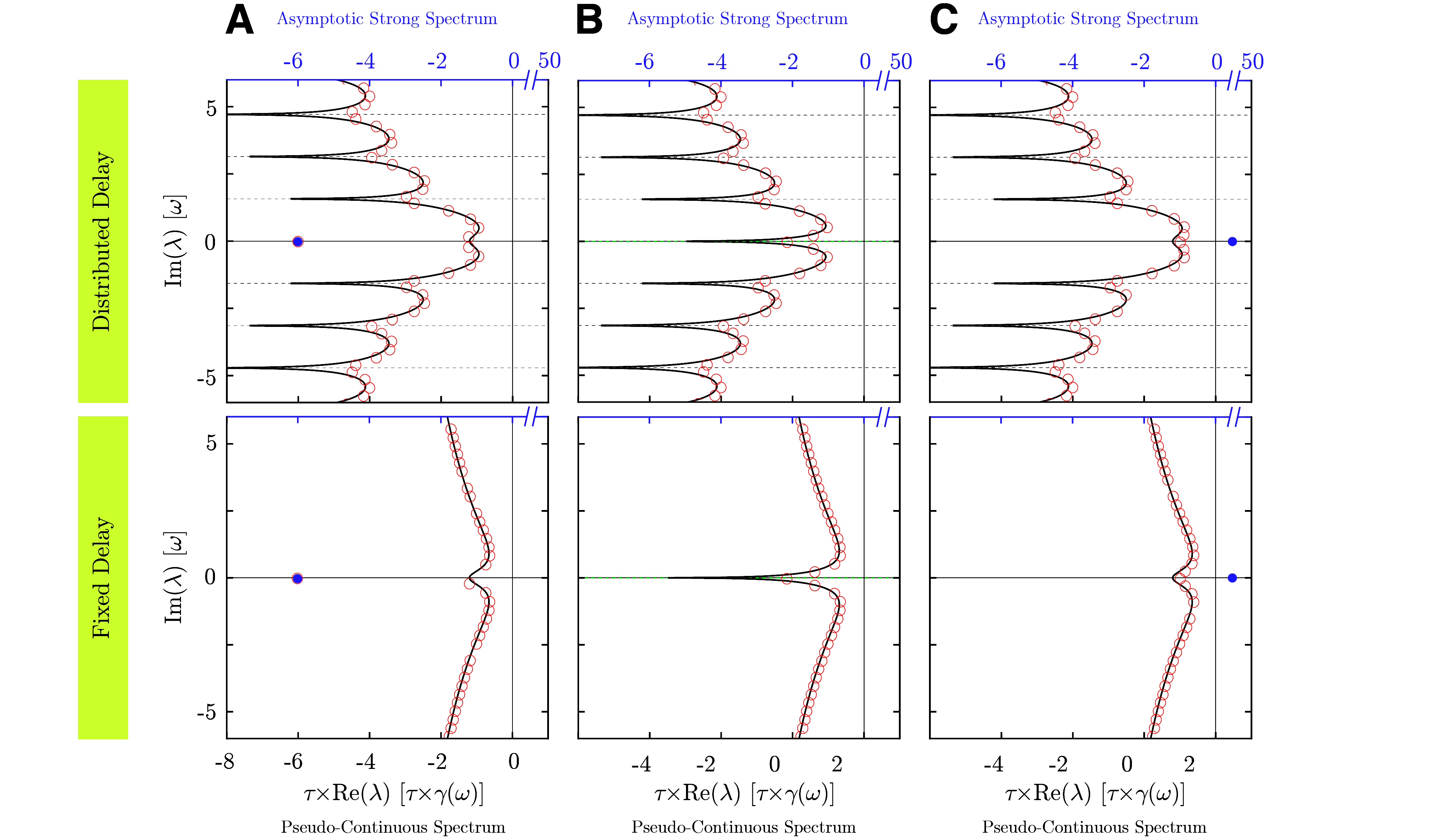}
       \caption{Comparison between uniform distributed delay  in Example \ref{example_2} and fixed delay in  \cite[Figure 3.2]{lichtner2011spectrum}. Parameter values are $\gavg=20$  and $\rho=2$:       {\rm (A)} $\alpha=-0.3$   {\rm (B)}    $\alpha=0$  {\rm (C)}    $\alpha=0.3$. 
Black curve is the pseudo-continuous spectrum, blue dots are the asymptotic strong spectrum, red circles are numerically computed eigenvalues, gray dashed lines are horizontal asymptotes at $\omega= k\pi/\rho$, where $k\in \mathbb{Z}\setminus\{0\}$, and green dashed line is horizontal asymptote at $\omega= 0$ when $\alpha=0$.}
\label{Fig:Example_2_2}
\end{figure}

 
\begin{example}\label{example_3}
    Consider the system 
    \begin{equation}\label{example3_eq}
    \frac{du}{dt}=\left[\begin{array}{cc}
\alpha & \beta \\
-\beta & \alpha
\end{array}\right]\,u(t)+ \frac{1}{2\rho} \,\int_{\tau_m-\rho}^{\tau_m+\rho} 
\left[\begin{array}{ll}
1 & 0 \\
0 & 1
\end{array}\right]\,u(t-s) \, ds
\end{equation}
where $u(t)=[u_1(t),u_2(t)]^T$ and $\alpha,\,\beta\in\mathbb{R}$. 
\end{example}
It is clear that
\[A=A_4=\left[\begin{array}{cc}
\alpha & \beta \\
-\beta & \alpha
\end{array}\right]
\qquad\text{and}\qquad
B=\widebar{B}=\left[\begin{array}{ll}
1 & 0 \\
0 & 1
\end{array}\right]
\]
Consequently, $\mathcal{A}_{-}=\emptyset$ and 
\[
\mathcal{A}_{s}=\mathcal{A}_{+}=\{\det(\lambda\,I_2-A)=0 \, \mid   \, \Re(\lambda)>0\}= \left\{ {\begin{array}{*{20}{c}}
{  \alpha\pm i\beta  ,}&{\alpha > 0,}\\
\emptyset,&{\alpha < 0.}
\end{array}} \right.\]
The characteristic equation corresponding to \eqref{example3_eq} is 
\[\Delta_3(\lambda)=\left(\lambda -\alpha-e^{-\lambda\tau_m}\,\sinhc( \rho\lambda)\right)^2-\beta^2=0.
\]
Consequently, we define 
\[
p_\omega(x)=\left(i\omega -\alpha-x\,\sinc( \rho\omega)\right)^2-\beta^2.
\]
Since  $\text{rank}(B)=2$,  there exist two spectral curves
\[  \gamma_{\pm}(\omega)=-\frac{1}{2}\ln\left(\frac{   \alpha ^2+(\beta \pm\omega )^2 }{\sinc ^2(\rho  \omega )}\right)  \]
that  satisfy
$p_\omega\left(e^{-(\gamma (\omega)+i\varphi)}\right)=0$ for some $\varphi\in\mathbb{R}$. Consequently, we have
\[
\mathcal{A}_c=
\left\{\gamma_{+}(\omega)+i \omega \,\mid\, \omega \in \mathbb{R},\, \omega\neq-\beta  \text{ or }  \alpha \neq 0\right\} 
\cup
\left\{\gamma_{-}(\omega)+i \omega \,\mid\, \omega \in \mathbb{R},\,\omega\neq  \beta  \text{ or }  \alpha \neq 0\right\}
\]
Figure \ref{Fig:Example_3_1} shows the accuracy of the pseudo-continuous spectrum predicting the numerically computed eigenvalues with different values of $\rho$ and fixed large delay $\gavg$. We observe the same behavior as in Examples \ref{example_1} and \ref{example_2}.
  
Note that when $\alpha=0$, the behavior of $\gamma_{ \pm}(\omega)$ depends on the value of $\omega$:
\begin{itemize}
    \item When $\omega= \beta$ with $\beta \neq k \pi / \rho$, then 
             \[\lim_{\omega \rightarrow  \pm \beta}\gamma_{\pm}(\omega ) = \ln \left(\frac{|\sin \left(\rho\beta \right)|}{2\rho\beta^2} \right)\qquad\qquad\text{and}\qquad\qquad
            \lim_{\omega \rightarrow  \pm \beta}\gamma_{\mp}(\omega ) =\infty, \]
            see Figure \ref{Fig:Example_3_3}A-B;
    \item    When $\omega= \beta$ with $\beta =k \pi / \rho$, $k\in\mathbb{Z}\setminus \{0\}$, then
             \[\lim_{\omega \rightarrow  \pm \beta}\gamma_{\pm}(\omega ) =-\infty \qquad\qquad\text{and}\qquad\qquad
            \lim_{\omega \rightarrow  \pm \beta}\gamma_{\mp}(\omega ) =\ln \left(\frac{\rho}{|k|\, \pi}\right), \]
            see Figure \ref{Fig:Example_3_3}C. Note that in this case the strong critical spectrum is nonempty: $\mathcal{A}_0=\{\pm i\,k\pi/\rho \}$.
\end{itemize}

For fixed delay, system \eqref{example3_eq} becomes
   \begin{equation}\label{example3_eq_fixed}
    \frac{du}{dt}=\left[\begin{array}{cc}
\alpha & \beta \\
-\beta & \alpha
\end{array}\right]\,u(t)+  
\left[\begin{array}{ll}
1 & 0 \\
0 & 1
\end{array}\right]\,u(t-\gavg)  
\end{equation}
Here, we have 
\[ \gamma_{fixed,\pm}(\omega)=-\frac{1}{2}\ln\left(\alpha ^2+(\beta \pm\omega )^2\right)\ge \gamma_\pm(\omega). \]
The authors in \cite{lichtner2011spectrum} showed that system \eqref{example3_eq_fixed} is stable if $|\alpha|>1$ and unstable if $|\alpha|<1$. 
This condition does not hold in the case of distributed delays, and stability also depends on $\rho$. In Figure \ref{Fig:Example_3_2}, we take $\alpha=-0.5$ and system  \eqref{example3_eq_fixed} corresponding to $\rho\to 0$ is unstable (Figure \ref{Fig:Example_3_2}A). As $\rho>0$ increases, the corresponding system \eqref{example3_eq} remains unstable for relatively small $\rho$ (Figures \ref{Fig:Example_3_2}B-\ref{Fig:Example_3_2}C), and then becomes stable for relatively large $\rho$ (Figures \ref{Fig:Example_3_2}D-\ref{Fig:Example_3_2}E).
Furthermore, in Figure \ref{Fig:Example_3_4}, we compare the results with \cite[Figure 3.1]{lichtner2011spectrum}. 
In both cases, we notice that the pseudo-continuous spectrum is stable (resp. unstable) when $|\alpha|>1$ (resp. $|\alpha|<1$), and if $\alpha>0$ the strong unstable spectrum is nonempty, see Figures \ref{Fig:Example_3_4}A and \ref{Fig:Example_3_4}E (resp. Figures \ref{Fig:Example_3_4}B and \ref{Fig:Example_3_4}D).  
Furthermore, when $\alpha=0$, a horizontal asymptote at $\omega=0$ exists. 
The two cases show the same behavior, but the distributed delay has an infinitely countable number of horizontal asymptotes while only one horizontal asymptote exists at $\omega=0$ when $\alpha=0$ in the fixed delay case.

\begin{figure} [!htb]
    \centering
    \includegraphics[width=1\linewidth]{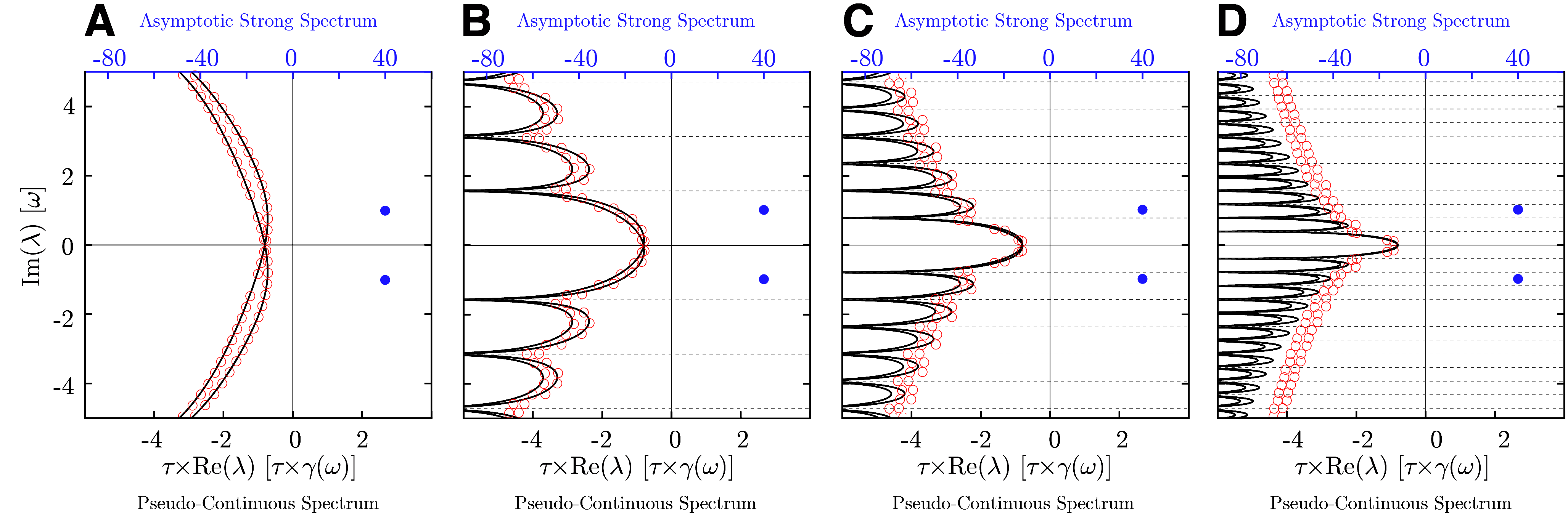}
   \caption{The effect of increasing $\rho$ with a fixed large mean delay $\gavg$  on the asymptotic calculations  in Example \ref{example_3}. Spectra of Example \ref{example_3} for large mean delay with $\gavg=20$, $\alpha=2$, and $\beta=1$: 
 {\rm (A)} $\rho=0.5$    {\rm (B)}  $\rho=2$ {\rm (C)}  $\rho=4$   {\rm (D)}  $\rho=8$. 
Black curve is the pseudo-continuous spectrum, blue dots are the asymptotic strong spectrum, red circles are numerically computed eigenvalues, and gray dashed lines are horizontal asymptotes at $\omega= k\pi/\rho$, where $k\in \mathbb{Z}\setminus\{0\}$.}
\label{Fig:Example_3_1}
\end{figure}

\begin{figure} [!htb]
    \centering
    \includegraphics[width=1\linewidth]{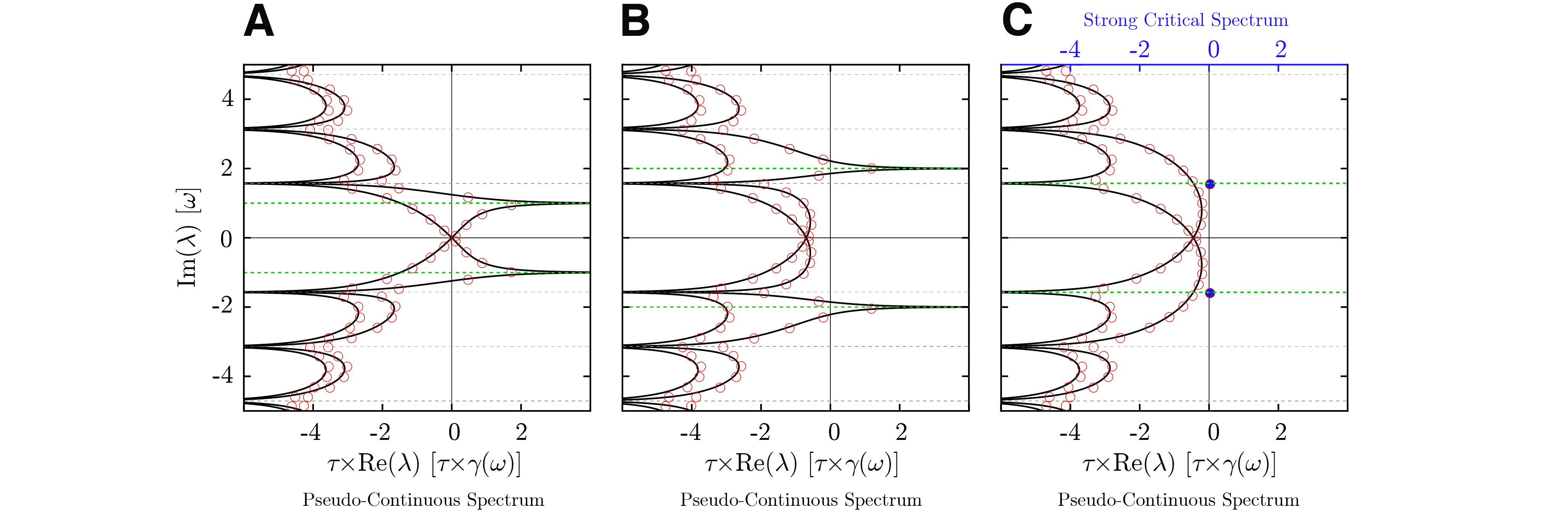}
   \caption{Spectra of Example \ref{example_3} for large mean delay $\gavg$ and $\alpha=0$.  Parameter values are   $\gavg=20$, $\rho=2$, and $\alpha=0$: 
 {\rm (A)}  $\beta=1$     {\rm (B)}  $\beta=2$   {\rm (C)}  $\beta=\pi/\rho=\pi/2$. 
Black curve is the pseudo-continuous spectrum, blue dots are the asymptotic strong spectrum, red circles are numerically computed eigenvalues, gray dashed lines are horizontal asymptotes at $\omega= k\pi/\rho$, where $k\in \mathbb{Z}\setminus\{0\}$, and green dashed lines are horizontal asymptotes at $\omega=  \pm \beta$.}
\label{Fig:Example_3_3}
\end{figure}

\begin{figure} [!htb]
    \centering
    \includegraphics[width=1\linewidth]{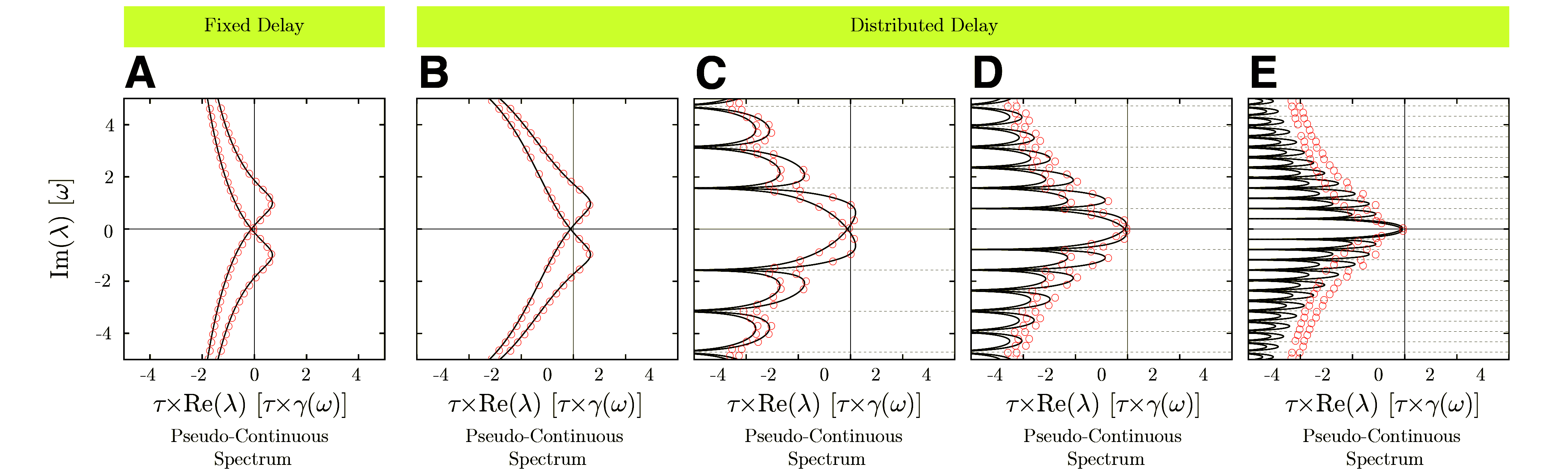}
      \caption{The effect of increasing $\rho$ with a fixed large mean delay $\gavg$ on the stability of system \eqref{example3_eq}  in Example \ref{example_3}. Spectra of Example \ref{example_3} for large mean delay with  $\gavg=20$, $\alpha=-0.5$, and $\beta=1$: 
{\rm (A)} Fixed delay  $\rho\to 0$  {\rm (B)}  $\rho=0.5$    {\rm (C)}  $\rho=2$  {\rm (D)}  $\rho=4$     {\rm (E)}  $\rho=8$. 
Black curve is the pseudo-continuous spectrum, blue dots are the asymptotic strong spectrum, red circles are numerically computed eigenvalues, and gray dashed lines are horizontal asymptotes at $\omega= k\pi/\rho$, where $k\in \mathbb{Z}\setminus\{0\}$.}
\label{Fig:Example_3_2}
\end{figure}

\begin{figure} [!htb]
    \centering
    \includegraphics[width=1\linewidth]{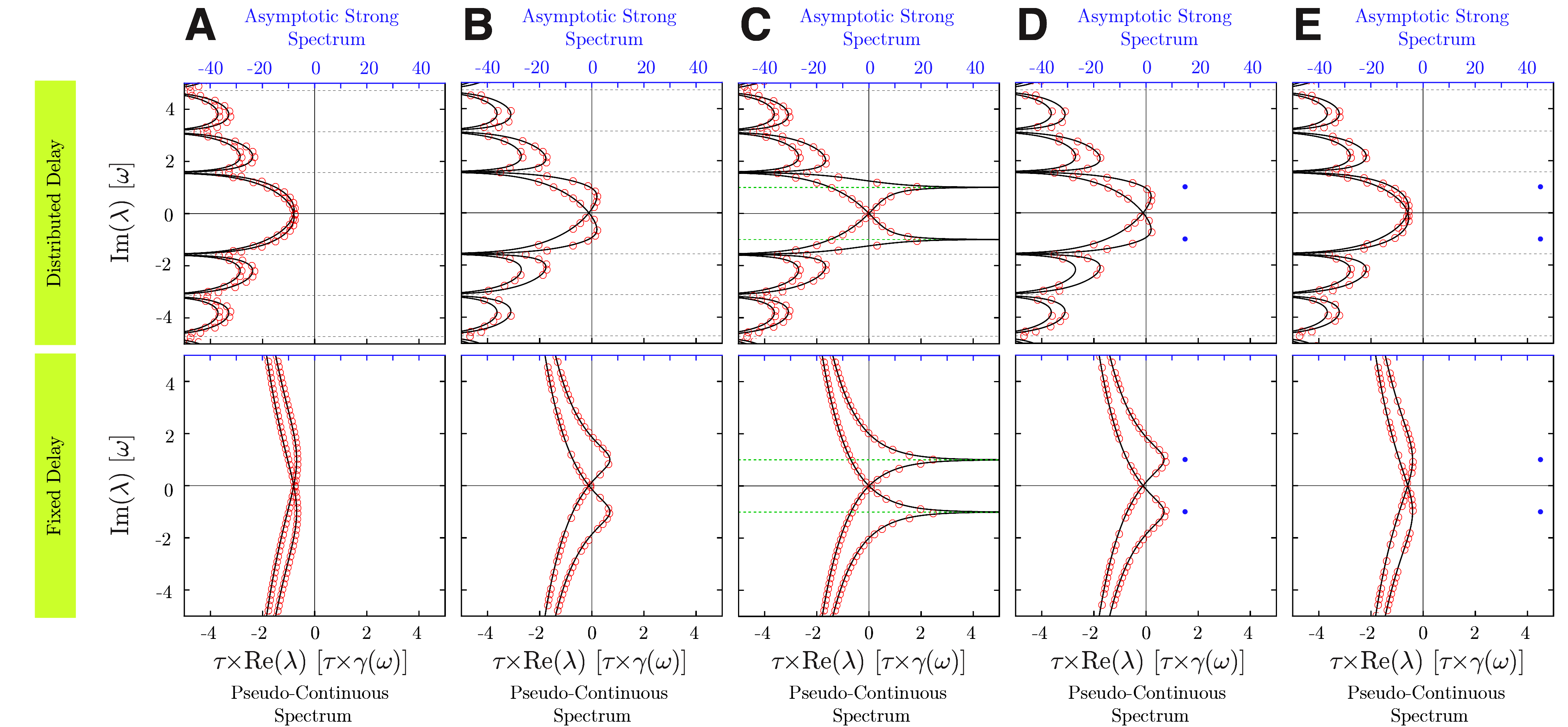}
  \caption{Comparison between uniform distributed delay  in Example \ref{example_3} and fixed delay in  \cite[Figure 3.1]{lichtner2011spectrum}. Parameter values are $\gavg=30$, $\rho=2$, and $\beta=1$:       
{\rm (A)} $\alpha=-2$   
{\rm (B)}    $\alpha=-0.5$  
{\rm (C)}    $\alpha=0$
 {\rm (D)}    $\alpha=0.5$
  {\rm (E)}    $\alpha=1.5$.
Black curve is the pseudo-continuous spectrum, blue dots are the asymptotic strong spectrum, red circles are numerically computed eigenvalues, gray dashed lines are horizontal asymptotes at $\omega= k\pi/\rho$, where $k\in \mathbb{Z}\setminus\{0\}$, and green dashed line is horizontal asymptote at $\omega=\pm \beta$ when $\alpha=0$.}
  
\label{Fig:Example_3_4}
\end{figure}

\section{Application: Wilson-Cowan Model with Delayed Self-coupling and Homeostatic Plasticity}\label{section_model}
Consider a popluation of neurons which consists of two sub-populations: one excitatory and one inhibitory. The Wilson-Cowan (WC) model is a simple yet effective  model to study the neural activity in this situation \cite{wilson1972,wilson2021evolution}. The model consists of two variables, $E$ and  $I$, representing the activity of the excitatory and inhibitory of neurons, respectively, where both variables take values in the interval $[0,1]$. Recent work has considered an extension to this model where the synaptic weight, $W^{EI}>0$, from the inhibitory to the excitatory population is adjusted to keep the activity of the excitatory population from being too large or too small \cite{hellyer2016local,nicola2018chaos}. This is called homeostatic plasticity. In previous work, we have studied networks of such Wilson Cowan nodes to represent large scale brain networks. This prior work primarily focussed on the effect of (small) time delay between nodes on stability and synchronization behaviour \cite{al2021impact,al2024distributed}. Here, we apply the results of Section~\ref{section_model} to study the spectrum and resulting dynamics in a single WC node  with uniformly distributed self-coupling delay. Solutions of this model correspond to solutions in the large brain network model where the nodes are synchronized.
The model we consider is as follows
\begin{eqnarray}\label{model}
	 \tau_2\frac{dW^{EI}}{dt} & = & I(E-p ),\nonumber\\[2mm]
	 \frac{dI}{dt} & = & \phi(W^{IE}E)  -I, \\[2mm]
\tau_1\frac{dE}{dt} & = & \phi \left(\frac{W^{E}}{2\rho} \int_{\gavg-\rho}^{\gavg+\rho} E(t-s)\, ds- W^{EI}I \right)-E. \nonumber
\end{eqnarray}
The parameters $\tauA>0$ and $\tauB>0$ are the ratios of the timescales of the excitatory population activity and of the plasticity of the inhibitory  synaptic weight, respectively, to that of the inhibitory population. The parameter $p$ is the homeostatic set-point for the activity of the excitatory population and thus satisfies $0\le p \le 1$.
The function  $\phi$  is the  transfer function which determines the proportion of the population of neurons which is active in node $i$. It is assumed to be sigmoidal, thus is strictly increasing and satisfies $0\le \phi(x)\le 1$.
Following \cite{al2021impact,nicola2021,nicola2018chaos}, we use the logistic function	
	\begin{eqnarray}\label{eq3}
	\phi(x)=\frac{1}{1+e^{-ax}}
		\end{eqnarray}	
where $a$ controls the steepness of the sigmoid.

Note that the full system has the equilibrium point $(W^{EI},I,E)=(W^{EI*},I^*,E^*)$ with
\begin{eqnarray}\label{eq:eqmdef}
W^{EI*}=\frac{W^Ep-\phi^{-1}(p)}{I^*}, \quad I^*=\phi(W^{IE}p), \quad E^*=p.
\end{eqnarray}
Let $u=(W^{EI},I,E)^T-(W^{EI*},I^*,E^*)^T$. Then, the linearization of model \eqref{model} about the equilibrium point is
\begin{equation}\label{model_linear}
    \frac{du}{dt}=A\,u(t)+\frac{1}{2\rho}B\,\int_{\gavg-\rho}^{\gavg+\rho} u(t-s)\, ds
\end{equation}
where
\begin{equation*} \arraycolsep=1.5pt \def\arraystretch{1.3}
A=
\left[
\begin{array}{c|c}
	{\arraycolsep=1.0pt \def\arraystretch{1.0}\begin{array}{ccc} && \\ & A_1 & \\&&   \end{array}}& A_2 \\
 \hline
	A_3 & A_4   
\end{array}
\right]
=
\left[
\begin{array}{cc|c}
	0& 0 &a_{13} \\[5pt]
	0 & -1 & a_{23} \\[8pt]\hline
	a_{31} & a_{32} & a_{33}   
\end{array}
\right]
=
\left[
\begin{array}{cc|c}
	 0 & 0   	&\frac{1}{\tauB}I^* \\[5pt]
	 0 & -1 & W^{IE}\phi'(W^{IE}p)  \\[8pt]\hline
-\frac{1}{\tauA}\phi'(\phi^{-1}(p))I^*  & -\frac{1}{\tauA}\phi'(\phi^{-1}(p))W^{EI*} &-\frac{1}{\tauA}
\end{array}
\right]
\end{equation*}
and
\begin{equation*} 
B=
\left[
{\arraycolsep=1.4pt \def\arraystretch{1.4}
\begin{array}{c|c}
{\arraycolsep=1.0pt \def\arraystretch{1.0}
 \begin{array}{ccc}
 ~&~&\,\,\, \\ 
 ~& \,0\, & \,\,\, \\
 ~&~&\,\,\,   
 \end{array}}& 0 \\ 
 \hline
	0 &  \widebar{B}
\end{array}}
\right]
=
\left[
\begin{array}{cc|c}
	0& 0 &0\\[5pt]
	0 & 0 & 0 \\[8pt]\hline
	0 & 0 & b_{33}   
\end{array}
\right]
=
\left[
\begin{array}{cc|c}
	0 & 0 & 0  \\[5pt]
	0& 0 & 0   \\[8pt]\hline
 0&0&-\frac{1}{\tauA} W^{IE}\phi'(W^{IE}p)
\end{array}
\right].
\end{equation*}

\begin{theorem}\label{Th_1}
    The asymptotic strong spectrum of \eqref{model_linear} is $\mathcal{A}_s=\{-1\}$.
\end{theorem}
\begin{proof}
    First, we calculate the asymptotic strong unstable spectrum $\mathcal{A}_+$. To this end, we have
   \begin{equation}\label{A_unstable}
       \det(\lambda I_3-A)=\lambda^3+p_2\lambda^2+p_1 \lambda+p_0=0
   \end{equation}
where
\begin{align*}
    p_2&=1-a_{33}=1+\frac{1}{\tauA}>0,\qquad 
    p_1=-(a_{33}+a_{13}a_{31}+a_{23}a_{32})=\cfrac{1}{\tauA}+\cfrac{{W}^{EI*}W^{IE}\phi'(\phi^{-1}(p)) \,\phi'(W^{IE}p)}{\tauA} +p_0>0,
   \\
 p_0&=-a_{13}a_{31}=\cfrac{{I^*}^2\phi'(\phi^{-1}(p))}{\tauA\tauB}>0.
\end{align*}
Moreover
\[
p_2p_1-p_0=\frac{1}{\tauA}\left(p_1+1+{W}^{EI*}W^{IE}\phi'(\phi^{-1}(p)) \,\phi'(W^{IE}p) \right)>0.
\]
Hence, all roots in equation \eqref{A_unstable} have a negative real part by the Routh-Hurwitz criterion \cite{levine2018control}. Therefore, $\mathcal{A}_+=\emptyset$.

It is easy to check that $\sigma(A_1)=\{0,-1\}$.
Hence, $\mathcal{A}_-=\{-1\}$. Consequently, $\mathcal{A}_s=\{-1\}$.
  \end{proof}
  
Theorem \ref{Th_1} implies that for sufficiently large delay, as $\gavg\to\infty$, there exist eigenvalues approaching $-1$ while no eigenvalue approaches any values with a positive real part. From the proof of Theorem~\ref{Th_1},  $\sigma(A)$ contains only eigenvalues with negative real part. Thus  $\mathcal{A}_0=\emptyset$. 
\begin{theorem}
    The  asymptotic continuous spectrum of \eqref{model_linear} is 
 \[\mathcal{A}_c=\left\{\gamma(\omega)+i \omega \in \mathbb{C} \mid  \omega \in \mathbb{R},  \,\,  \rho\omega\neq k\pi\,\text{for}\,k\in \mathbb{Z}\right\}\] 
    where
      \begin{equation}\label{eq_1111}
       \gamma(\omega)=-\frac{1}{2}\ln\left(  \frac{(p_2\omega^2-p_0)^2+\omega^2(\omega^2-p_1)^2}{q^2\,\omega^2\,(1+\omega^2)\,\sinc^2(\rho\omega)}  \right) 
     \qquad\text{with}\qquad q=\cfrac{W^E\phi'(\phi^{-1}(p))}{\tauA}>0.     
      \end{equation}
\end{theorem}
\begin{proof}
First, note that  the characteristic equation associated with \eqref{model_linear} is
\[
\Delta_4(\lambda)=\lambda^3+p_2\lambda^2+p_1 \lambda+p_0-q\lambda(\lambda+1)\,e^{\lambda\gavg}\sinhc(\rho\lambda)=0.
\]
Hence, we obtain
\[p_{\omega}(x)=-i\omega^3-p_2\omega^2+ip_1\omega+p_0-q\omega(i-\omega)\,\sinc(\rho\omega)\,x.\]
 Assume $\rho\omega\neq k\pi$ for $k\in \mathbb{Z}\setminus\{0\}$.
 Since $\text{rank}(B)=1$ and $i\omega\notin\sigma(\mathcal{A}_1)$, we have only one spectral curve $\gamma(\omega)$ satisfying $p_\omega\left(e^{-(\gamma(\omega)+i\varphi)}\right)=0$. 
    By separating the real and imaginary parts of this equation, we obtain
  \begin{neweq_non} 
 p_2 \omega^2-p_0&= \omega\, q\, \sinc(\rho \omega)\, e^{-\gamma(\omega)}\, \left[\omega\cos(\varphi)-\sin(\varphi)  \right] \\
\omega^3-p_1\omega&=-\omega\, q\, \sinc(\rho \omega)\, e^{-\gamma(\omega)}\, \left[\omega\sin(\varphi)+\cos(\varphi)  \right] .
\end{neweq_non}  
  Squaring and adding both equations lead to 
  \[( p_2 \omega^2-p_0)^2+(\omega^3-p_1\omega)^2=\omega^2\, q^2\, \sinc^2(\rho \omega)\, e^{-2\gamma(\omega)}\,\]
 Consequently, we have
      \begin{equation}\label{gamma_fun_e}
      \gamma(\omega)=-\frac{1}{2}\ln\left(  \frac{(p_2\omega^2-p_0)^2+\omega^2(\omega^2-p_1)^2}{q^2\,\omega^2\,(1+\omega^2)\,\sinc^2(\rho\omega)}  \right).    
      \end{equation}
Since $q>0$ and the numerator inside the log function is positive, we have $ \gamma(\omega)\to -\infty$ if  $\omega\rho=k\pi$ for $k\in \mathbb{Z}$. This completes the proof.  
\end{proof}

In the following, we set the model \eqref{model} parameter values as in \cite{nicola2018chaos, nicola2021, al2021impact, al2024distributed}:  
\begin{equation}  
p = 0.2, \quad a = 5, \quad \tau_1 = 1, \quad \tau_2 = 5.  
\label{params}  
\end{equation}  
These values are used to plot the pseudo-continuous spectrum curves. We also numerically computed the eigenvalues, demonstrating consistency with the theoretical results and analyzing the effect of $\rho$ on stability.

In Figures \ref{Fig:model_2} and \ref{Fig:model_1}, we notice, similar to the previous examples, that the pseudo-continuous spectrum closely tracks the numerically computed eigenvalues. Figure \ref{Fig:model_2} shows that the equilibrium point becomes stable as $\rho$ increases. On the other hand, Figure \ref{Fig:model_1} explores the influence of the parameter $W^{IE}$ on the stability of the equilibrium point. We notice that the equilibrium point starts off unstable when $W^{IE}$ is small. Then, as $W^{IE}$ increases the equilibrium point gains stability, and then becomes unstable again. Numerical simulations of the full nonlinear system \eqref{model} confirm the stability predictions of the linearization and show that the system transitions to stable oscillations when stability is lost, likely via a Hopf bifurcation.
 
\begin{figure} [!htb]
    \centering
    \includegraphics[width=1\linewidth]{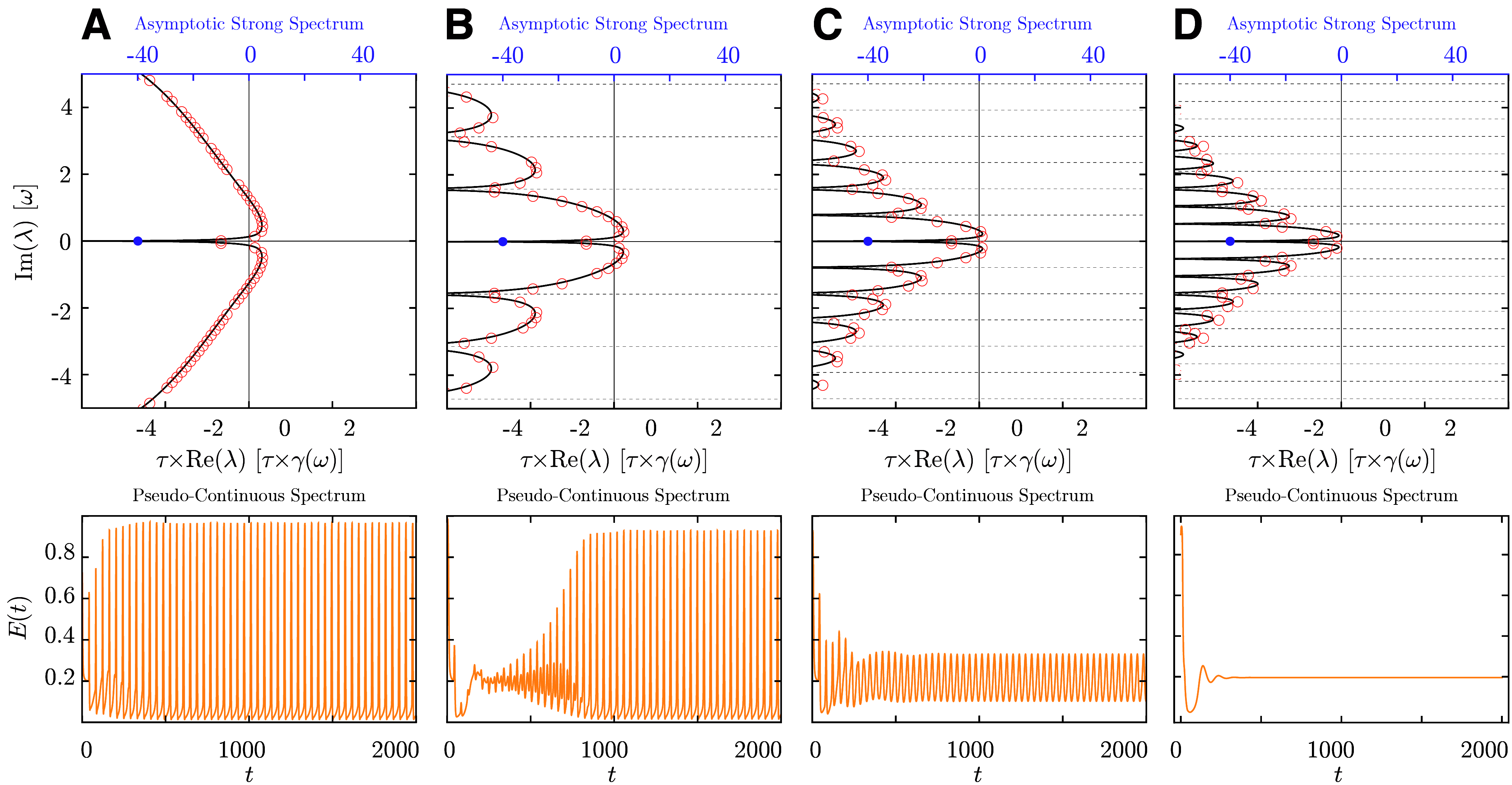}
       \caption{The effect of increasing $\rho$ with a fixed  large $\gavg$. 
    Row 1: Spectra of system \eqref{model_linear}.  
     Row 2: Time series for the solutions  of system \eqref{model}. 
Parameter values are  $\gavg=40$, $W^E=2$, and $W^{IE}=4$: {\rm (A)} $\rho=0.5$     {\rm (B)} $\rho=2$  {\rm (C)} $\rho=4$    {\rm (D)} $\rho=6$. 
Black curve is the pseudo-continuous spectrum, blue dots are the asymptotic strong spectrum, red circles are numerically computed eigenvalues, and gray dashed lines are horizontal asymptotes at $\omega= k\pi/\rho$, where $k\in \mathbb{Z}\setminus\{0\}$.}
\label{Fig:model_2}
\end{figure}

\begin{figure} [!htb]
    \centering
    \includegraphics[width=1\linewidth]{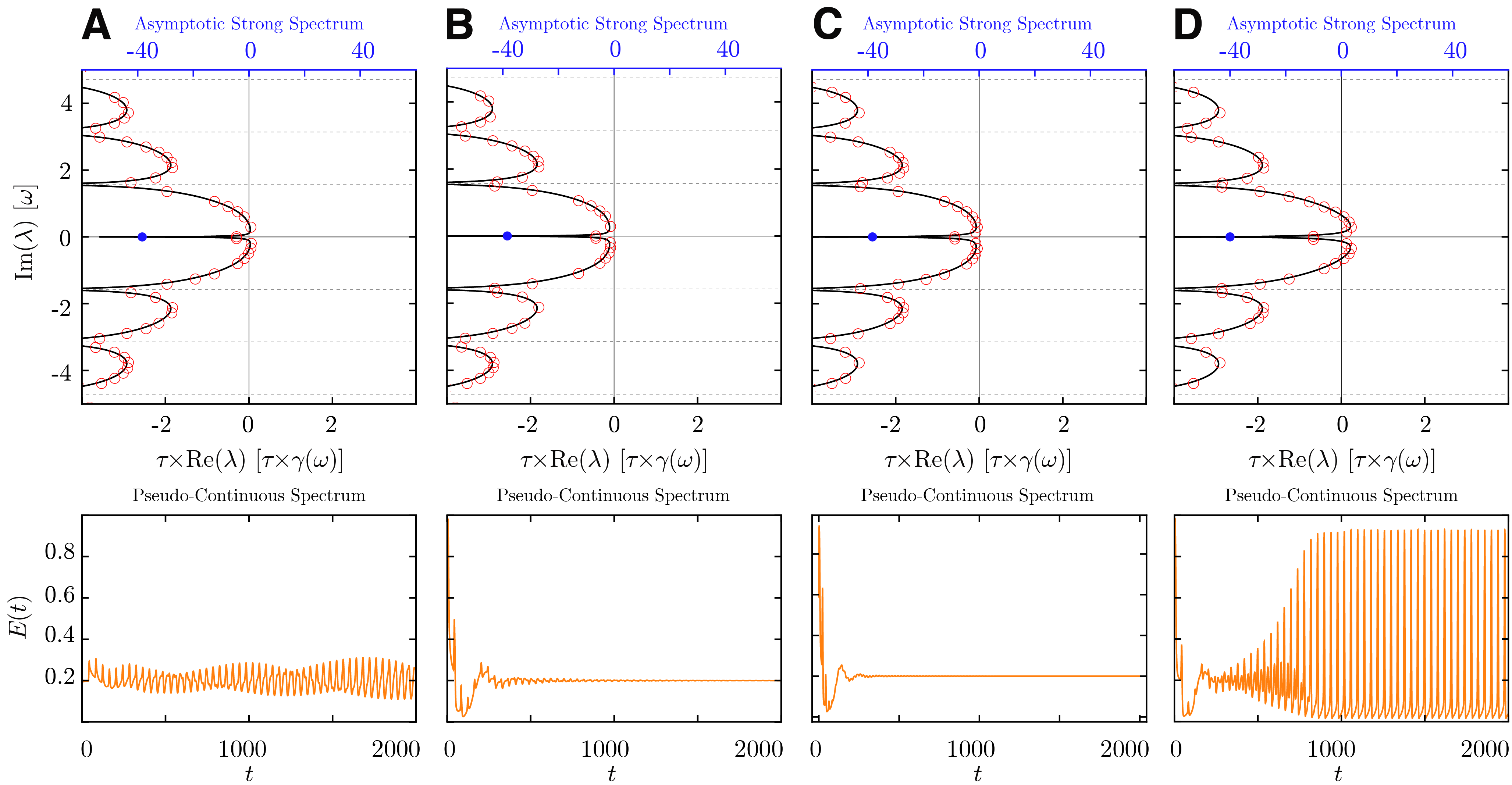}
    \caption{The effect of increasing $W^{IE}$ with fixed $\rho$ and large mean delay  $\gavg$. 
    Row 1: Spectra of system \eqref{model_linear}.  
     Row 2: Time series for the solutions of system \eqref{model}. 
Parameter values are  $\gavg=40$, $\rho=2$, and $W^E=2$:  {\rm (A)}  $W^{IE}=0.5$      {\rm (D)}  $W^{IE}=1$       {\rm (C)}  $W^{IE}=2$ {\rm (D)}  $W^{IE}=4$. 
Black curve is the pseudo-continuous spectrum, blue dots are the asymptotic strong spectrum, red circles are numerically computed eigenvalues, and gray dashed lines are horizontal asymptotes at $\omega= k\pi/\rho$, where $k\in \mathbb{Z}\setminus\{0\}$.}
\label{Fig:model_1}
\end{figure}

To further explore the effect of parameters on the stability of the equilibrium point,
 we determine points in the $(\rho,W^{IE},W^E)$-space where $\gamma(\omega)$ changes from negative to positive.
Figure \ref{Fig:model_bif_3d}  shows a surface of such points in $(\rho,W^{IE},W^E)$-space. 
We generate this diagram as follows.
We use the \textit{Wolfram Mathematica} command \textsf{NMaximize} to  
calculate the absolute maximum value of $\gamma(\omega)$ in \eqref{eq_1111}
\[
\gamma^*=\max\{\gamma(\omega)\,:\,\omega\in\mathbb{R}\},
\]
on a grid over the  $(\rho,W^{IE},W^E)$-space. Then, we approximate the boundary of the regions where $\gamma^*$ changes its sign from negative to positive using the commands \textsf{ListInterpolation} and  \textsf{RegionPlot}. The surface shown in Figure~\ref{Fig:model_bif_3d} thus forms the boundary of the stability region of the equilibrium and each point on the surface is a (potential) Hopf bifurcation point. 
In Figure \ref{Fig:model_bif_2d}, we show the projections of the  3D diagram on $(W^{IE},W^E)$- and $(\rho,W^E)$-plane. The curves confirm that the increasing $\rho$ stabilizes the equilibrium point, while increasing $W^E$ destabilizes it and increasing $W^{IE}$ can be stabilizing or destabilizing.

The simulations in Figures~\ref{Fig:model_2}, \ref{Fig:model_1} and for other parameter values we have explored, show that the nonlinear system exhibits oscillations when the equilibrium point is unstable. This supports the assumption that the stability boundaries are curves of Hopf bifurcation. Interestingly, in Figures~\ref{Fig:model_2}B and ~\ref{Fig:model_1}D when the spectrum has multiple pairs of eigenvalues with positive real part the system only exhibits oscillations, although more complex transient behaviour is observed.

The stability boundaries/Hopf bifurcation curves in Figure~\ref{Fig:model_bif_2d}A and associated behaviour of the nonlinear similar are very similar to those we found for small delay in our previous work ~\cite{al2021impact,al2024distributed}. This led us to wonder if the behaviour was similar for all delays. Thus we considered the case where the mean delay is an order of magnitude larger than our previous work. 
Figure \ref{Fig:small_delay}A shows the Hopf curves using the asymptotic spectrum with the true spectrum in this case. Clearly the asymptotic spectrum is not accurate, which is not surprising as $\tau_m$ is not very large. More interestingly, the shape of the true spectrum is quite different, and the simulations of the full system (Figures \ref{Fig:small_delay}B-\ref{Fig:small_delay}E) show more complex behaviour than oscillation occurs.  Thus  although the behaviour of the system for large and small delays is basically the same, this is not true for all delays.

\begin{figure} [!htb]
    \centering
    \includegraphics[width=1\linewidth]{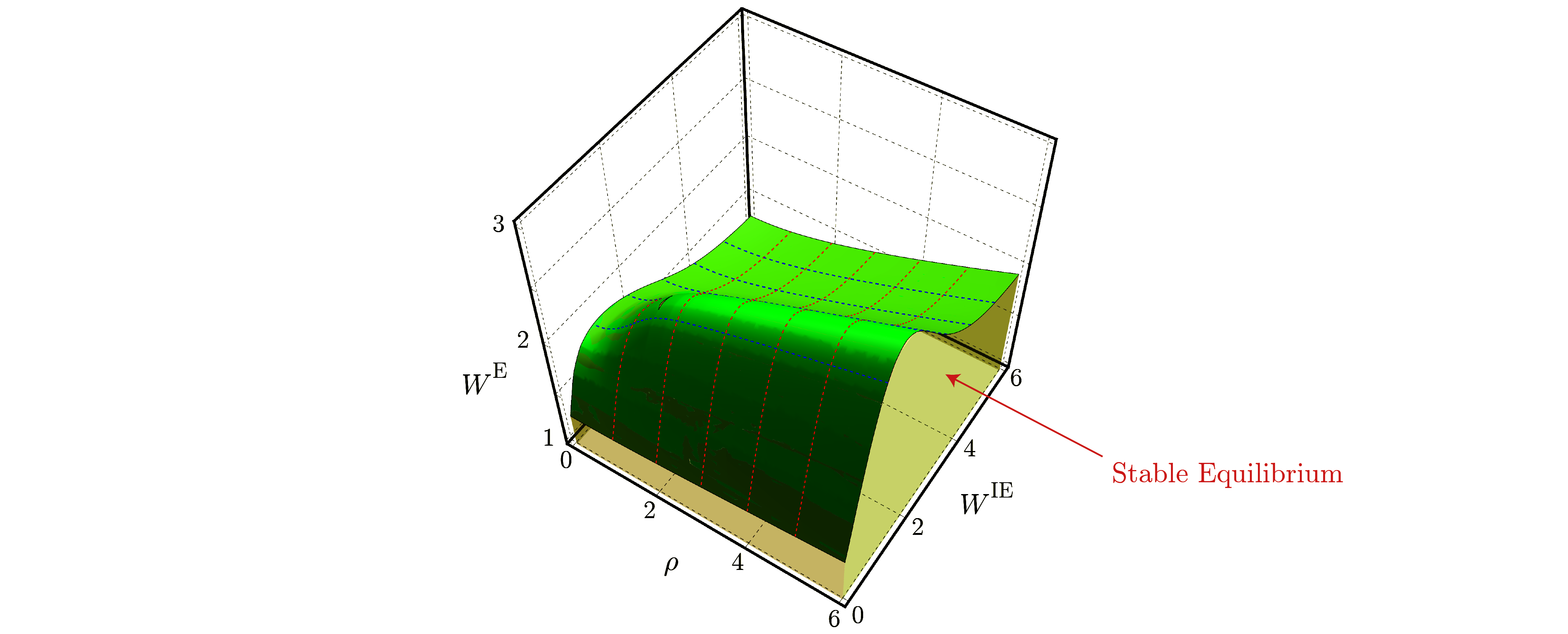}
       \caption{Three-parameter bifurcation diagram in $(\rho,W^{IE},W^E)$-space for large mean delay $\gavg\to\infty$. Below the surface the equilibrium point is asymptotically stable. Above the surface it is unstable. Points on the surface correspond to points where the equilibrium point has a pair of pure imaginary eigenvalues and are (potential) Hopf bifurcation points.}
\label{Fig:model_bif_3d}
\end{figure}

  \begin{figure} [!htb]
    \centering
    \includegraphics[width=1\linewidth]{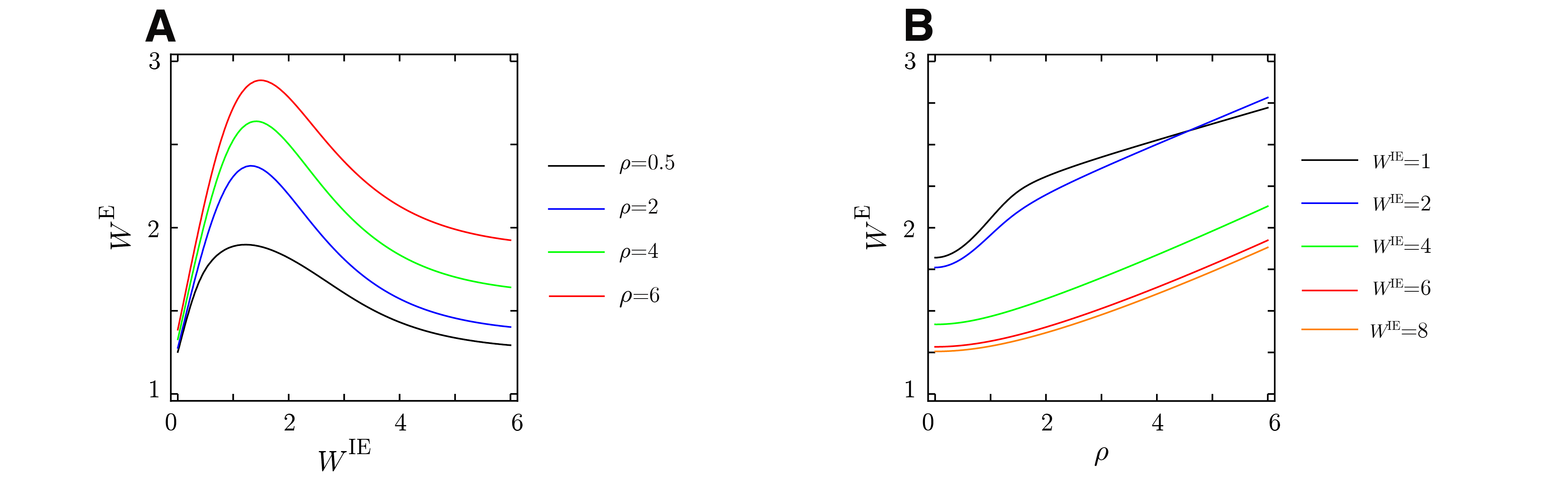}
       \caption{Bifurcation curves for large mean delay $\gavg\to\infty$.
        {\rm (A)} Hopf curves in the $(W^{IE},W^E)$ parameter space;  {\rm (B)} Hopf curves in the $(\rho,W^E)$ parameter space. The region under the curves corresponds to stability of the equilibrium, while the area above corresponds to instability.}
\label{Fig:model_bif_2d}
\end{figure}

\begin{figure} [!htb]
    \centering
    \includegraphics[width=1\linewidth]{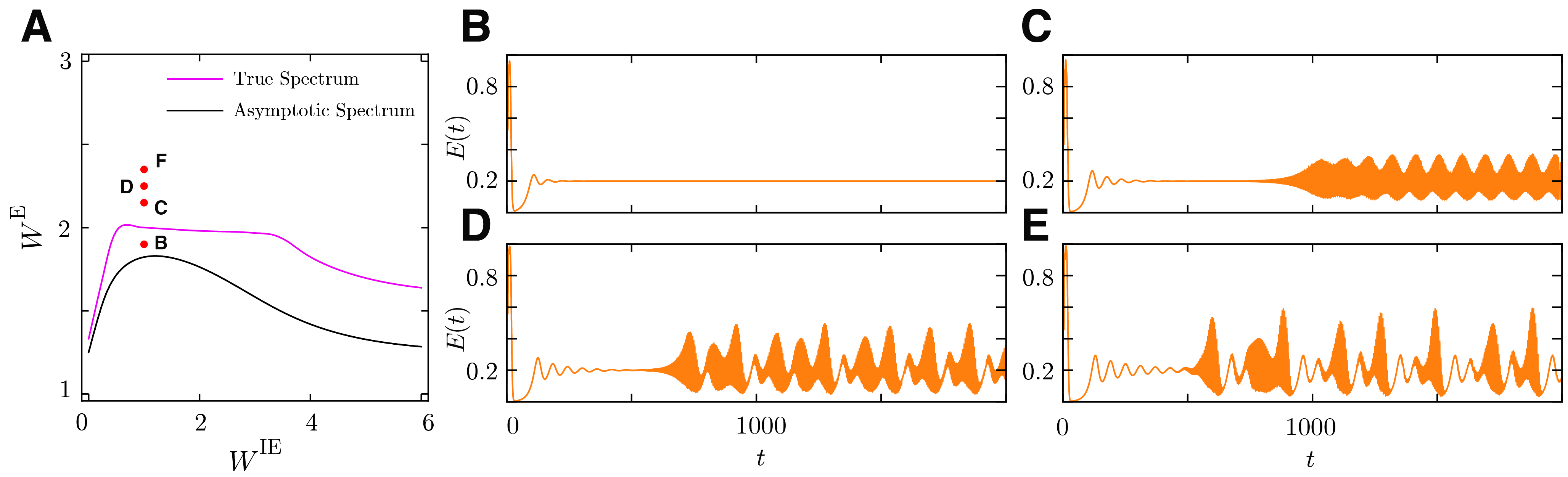}
       \caption{Comparison of the asymptotic spectrum with the true spectrum for small mean delay  $\gavg$. {\rm (A)} Hopf curves in the $(W^{IE},W^E)$ parameter space when  $\gavg=4$ and $\rho=0.1$;
       {\rm (B)-(E)} Time series for the solutions of system \eqref{model} with 
       $W^{IE}=1$ and  $W^E= 1.9,~2.15,~2.25,~2.35$, respectively.
       }
    \label{Fig:small_delay}   
\end{figure}

\section{Discussion}
\label{sec_discussion}

In this paper, we extended the results of \cite{lichtner2011spectrum} for linear systems with large fixed delay to those with a uniform distribution of delays. We showed that when the mean delay of the distribution is large enough the spectrum is a union of the following sets.
\begin{enumerate}
  \item[(i)] The strong critical spectrum $\mathcal{A}_0$, which has pure imaginary eigenvalues that do not depend on $\tau_m$.
    \item[(ii)] The strong spectrum $\mathcal{P}^\epsilon_s$,  which is a finite set of eigenvalues that approach some fixed value as  $\gavg\rightarrow\infty$. The set of limiting values of these eigenvalues is called the asymptotic strong spectrum, $\mathcal{A}_s$, which consists of two subsets: the asymptotic strong unstable spectrum, $\mathcal{A}_+$, and asymptotic strong stable spectrum, $\mathcal{A}_-$, that, respectively, contain eigenvalues with positive and negative real parts.   
 \item[(iii)] The pseudo-continuous spectrum $\mathcal{P}_c^{\epsilon}$, which has eigenvalues  with real parts that scale as $\epsilon=1/\gavg$, $0<\epsilon\ll 1$, that is, for $\lambda\in \mathcal{P}_c^{\epsilon}$, we have
\[\lambda=\epsilon \gamma+i\omega +O\left(\epsilon^2\right),\qquad \text{for}\,\gamma, \omega\in\mathbb{R},\gamma\ne 0\] 
Note that these eigenvalues all approach the imaginary axis as $\epsilon\rightarrow 0$, i.e., $\gavg\rightarrow\infty$.
\end{enumerate}

We gave explicit descriptions of the sets $\mathcal{A}_0$ and $\mathcal{A}_\pm$  and of the curves, $\gamma(\omega)$, representing the scaled real part of the eigenvalues of $\mathcal{P}_c^{\epsilon}$ in terms of the model parameters. The asymptotic strong unstable and stable spectra are identical to that for the system with a fixed delay. The critical strong spectrum may be different, in particular, $\mathcal{A}_0$ for the system with a fixed delay is a subset of  $\mathcal{A}_0$ for the same system with a uniform distributed delay.
We argued that a larger half-width of the distribution, $\rho$, may require a larger mean delay for the asymptotic spectra to be a good representation of the true spectra. Since the fixed delay corresponds to a distribution with $\rho=0$, the approximation may be better for a fixed large delay than a distribution with the same mean delay.

We applied our results to three simple linear models and to the linearization of a complicated nonlinear neural field model and found the following.
The asymptotic pseudo-continuous spectrum was considerably more complex for the system with distributed delay. In particular, the curves of the asymptotic pseudo-continuous spectrum always have a countable infinity of horizontal tangent lines, corresponding to $\omega=k\pi/\rho,\ k \in \mathbb{Z}\setminus\{0\}$, whereas the fixed-delay system has at most one horizontal tangent at $\omega=0$.

In the examples we explored, we found cases where some subset of the stability region is the same for the system with fixed and distributed delay, but more generally, the stability depends on the size of the distribution half-width. In some cases, increasing the distribution half-width can stabilize a system which is unstable with fixed delay. In the nonlinear example we studied, the asymptotic spectrum enabled investigation of potential bifurcations with respect to several parameters.

In this paper, we chose to focus on the uniform distribution. While the details of the proofs rely on this specific distribution, the main asymptotic arguments remain valid, and similar extensions should be possible for other distributions.
\section*{Acknowledgements}
IAD acknowledges the support of the Jordan University of Science and Technology through funding for a research visit to the University of Waterloo (grant number 20240164). SAC acknowledges the support of the Natural Sciences and Engineering Research Council of Canada. BR acknowledges the support of the UKH and CAMB.
\bibliographystyle{ieeetr}   
\bibliography{References.bib}
\appendix
\numberwithin{equation}{section}
\numberwithin{figure}{section}
\numberwithin{table}{section}

\setcounter{theorem}{0}
\renewcommand{\thetheorem}{A.\arabic{theorem}}
\setcounter{lemma}{0}
\renewcommand{\thelemma}{A.\arabic{lemma}}
\makeatletter
\def\@seccntformat#1{\@ifundefined{#1@cntformat}%
   {\csname the#1\endcsname\quad} 
   {\csname #1@cntformat\endcsname} 
}
\section{Appendix: Proofs of the Properties of the Asymptotic Spectrum}\label{sec:appendix}

The results in this Appendix follow those of \cite{lichtner2011spectrum} with modifications to account for the distributed delay and some rearrangements for clarity. For completeness and clarity we have reproduced the full proofs. The correspondence between the results presented in our paper and those in \cite{lichtner2011spectrum} is as follows
\begin{multicols}{2}
	\begin{itemize}
		\item Theorem 2 - Lemma 7
		\item Theorem 3 - Corollary 5
		\item Theorem 4 - Corollary 6
		\item Theorems A.1 and A.3 - Theorem 3
		\item Theorem A.2 - Theorem 4
		\item Lemma A.1 - Lemma 10
		\item Lemma A.2 - Lemma 9
		\item Lemma A.3 - Lemma 12
	\end{itemize}
\end{multicols}

\noindent Let $\epsilon=1/\gavg$. Consequently, the  characteristic equation \eqref{charact_eq_tau} can be written as 
\begin{equation}\label{charact_eq_epsilon}
	\Delta(\lambda;\epsilon):=\det\left( \lambda I_n-A-B\,e^{-\lambda/\epsilon}\,\sinhc(\lambda\rho)\right)=0
\end{equation}

\begin{remark} 
	Note that the original model \eqref{model_linear_discrete}-\eqref{unieq1} is well defined only if $\rho\le \gavg$, i.e., $\epsilon\le 1/\rho$. Thus, our results will include hypotheses to ensure this requirement is satisfied.
\end{remark}

We adopt the following definitions from \cite{lichtner2011spectrum}.  
Define the spectrum to be the set
\[\sigmaeps=\{ \lambda\in\mathbb{C} \, :   \, \Delta(\lambda;\epsilon)=0\}, \]
and the the strong spectrum to be the set
\[
\sigmaeps_s:=\sigmaeps_{+} \medcup \sigmaeps_{-},
\]
where
\[
\sigmaeps_{+}:=\sigmaeps \medcap \ball_r\left(\mathcal{A}_{+}\right), \quad \sigmaeps_{-}:=\sigmaeps\medcap \ball_r\left(\mathcal{A}_{-}\right)
\]
and 
\begin{equation}
	r:=\frac{1}{3} \min \left\{r_0, \operatorname{dist}\left(\mathcal{A}_s, i \mathbb{R}\right)\right\} 
	\qquad \text{with}\qquad  
	r_0:=\min \left\{|\lambda-\mu| \, :   \, \lambda, \mu \in \mathcal{A}_s, \lambda \neq \mu\right\}.
	\label{r_def}
\end{equation}
Here, we  denote 
\[\ball_r(M):=\bigcup_{x \in M}\{z \in \mathbb{C}\, :   \, |x-z |<r\}\] be the set of balls around the set $M\subset \mathbb{C}$. 

Finally, we define
the pseudo-continuous spectrum  
\[
\sigmaeps_c=\sigmaeps \backslash \left(\sigmaeps_s  \medcup \mathcal{A}_0\right).
\]
and the rescaled pseudo-continuous spectrum 
\[
\pieps_c:=\seps\left(\sigmaeps_c\right),
\]
where $S: \mathbb{C} \rightarrow \mathbb{C}$ is the rescaling
\begin{equation}
	\seps(a+i b):=a \frac{1}{\epsilon}+i b \quad \text { for } \quad a, b \in \mathbb{R}.
	\label{Sepsilon}
\end{equation}


The following theorem demonstrates that as $\gavg \to \infty$ (or equivalently, $\epsilon \to 0$),  the strong spectrum $\sigmaeps_s$   approaches $\mathcal{A}_s$.
\begin{theorem}\label{thm_append_1} 
	\begin{enumerate}[series=MyList,label= ({\roman*})]
		\item Let $\lambda \in \mathcal{A}_{+}$. For any $0 < \delta \leq r$, there exists $0<\epsilon_0\le 1/\rho$ so that whenever $0 < \epsilon < \epsilon_0$, the number of eigenvalues, counting multiplicity, of \eqref{model_linear_prel} within $\ball_\delta(\lambda)$ equals the multiplicity of $\lambda$ as an eigenvalue of $A$.
		
		\item  Let $\lambda \in \mathcal{A}_{-}$. For any $0 < \delta \leq r$, there exists $0<\epsilon_0\le 1/\rho$ so that whenever $0 < \epsilon < \epsilon_0$, the number of eigenvalues (counted with multiplicities) of \eqref{model_linear_prel} within $\ball_\delta(\lambda)$,  equals the multiplicity of $\lambda$ as an eigenvalue of $A_1$.
	\end{enumerate}
\end{theorem}

\begin{proof}
	
	(i)   Let $\lambda \in \mathcal{A}_{+}$ and $ \delta \in (0, r]$. First, we note that  $\ball_\delta(\lambda)$ is a connected open set and for the eigenvalue $z\in \ball_\delta(\lambda)$, the holomorphic function $\Delta(z;\epsilon)$ defined in \eqref{charact_eq_epsilon} converges uniformly to the polynomial $\widebar{\Delta}(z)=\det\left( z I_n-A\right )$ as $\epsilon\to0$ because  $\rm{Re}(z)>0$. 
	Let 
	$z_0$  be a root of $\widebar{\Delta}(z)$ of  multiplicity $m$. Then,    Hurwitz theorem, see e.g., \cite{gamelin2003complex}, implies that there exists a small enough $\epsilon_0\le 1/\rho$ such that if $0<\epsilon<\epsilon_0<\delta$ the function $\Delta(\lambda;\epsilon)$  has exactly $m$ zeros, counting multiplicity, in $\ball_{\epsilon_0}(z_0)\subset \ball_{\delta}(z_0)$.

	(ii) Let $\lambda \in \mathcal{A}_{-}$  and $ \delta \in (0, r]$. Note that $\sinhc(\rho z)\ne 0$ for $z\in \ball_{\delta}(\lambda) $. 
	For 
	$z\in \ball_{\delta}(\lambda)$, define using  \eqref{charact_eq_epsilon}, \eqref{matrix_B}, and \eqref{matrix_A} the function
	\begin{align*}
		\widehat{\Delta}_{\lambda}(z;\epsilon):=  
		\Delta(z+\lambda;\epsilon)&=\det\left( \begin{bmatrix}
			(z+\lambda) I_{n-d} - A_1 &  -A_2 \\
			-A_3 & (z+\lambda) I_d - A_4 - \widebar{B}\,e^{-(z+\lambda)/{\epsilon}}\,\sinhc((z+\lambda)\rho)
		\end{bmatrix} \right)\\
		&:=
		\det \left( \left[\begin{array}{c|c}
			P_1 & P_2 \\\hline
			P_3 & P_4
		\end{array}\right]\right).   
	\end{align*}
	It then follows by \cite[Proposition 2.8.4]{bernstein2009matrix} that 
	\[
	\widehat{\Delta}_{\lambda}(z;\epsilon)=\det(P_4) \det(P_1-P_2P_4^{-1}P_3)=\left(e^{-(z+\lambda)/{\epsilon}}\right)^d\,\sinhc^d((z+\lambda)\rho)\, \det(Q_1) \det(Q_2)
	\]
	where
	\begin{align*}
		Q_1&:=e^{(z+\lambda)/{\epsilon}}\,\cschc((z+\lambda)\rho)\, \left((z+\lambda) I_d-A_4\right)-\widebar{B}=e^{(z+\lambda)/{\epsilon}}\,\cschc((z+\lambda)\rho)P_4,, \\
		Q_2&:=(z+\lambda) I_{n-d}-A_1-e^{(z+\lambda)/{\epsilon}} \,\cschc((z+\lambda)\rho)\, A_2 Q_1^{-1} A_3.  
	\end{align*}
	Here, $\cschc(x)=1/\sinhc(x)$. 
	Note that $Q_1$ is invertible for sufficiently small $\epsilon$ ($\equiv$ sufficiently large $\gavg$) 
	because
	\[
	\det(Q_1)=\,\det(-\widebar{B})+O\left(\left|e^{(z+\lambda)/{\epsilon}}\right|\right)
	\]
	It then follows from 
	\[
	Q_1^{-1} =-\widebar{B} ^{-1}+O\left(\left| e^{(z+\lambda)/{\epsilon}}\right|\right)
	\]
	that
	\[
	\det(Q_2)=\det\left((z+\lambda) I_{n-d}-A_1\right)+O\left(\left| e^{(z+\lambda)/{\epsilon}}\right|\right).
	\]
	Finally, we have
	\[
	\widehat{\Delta}_{\lambda}(z;\epsilon)\,\cschc^d((z+\lambda)\rho)\,\left(e^{(z+\lambda)/{\epsilon}}\right)^d=\det(-\widebar{B}) \det\left((z+\lambda) I_{n-d}-A_1\right)+O\left(\left| e^{(z+\lambda)/{\epsilon}}\right|\right)
	\]
	Recall that  $\rm{Re}(\lambda)<0$ and $\det(\widebar{B})\ne0$,
	Then,  using the same arguments as in part (i),   we have that for any $0<\delta \leq r$, there exists $0<\epsilon_0\le 1/\rho$ so that whenever $0<\epsilon<\epsilon_0$,
	$\widehat{\Delta}_{\lambda}(z;\epsilon)$ has the same number of zeros, counting multiplicity,  as $\det\left((z+\lambda) I_{n-d}-A_1\right)$ in $\ball_{\delta}(\lambda)$.
\end{proof}


In the following result, we provide some properties of the asymptotic continuous spectrum  $\mathcal{A}_c$ and the singularities of the pseudo-continuous spectrum curve. 
\begin{theorem}\label{theorem_rescaled_pseudo_continuous}
	\begin{enumerate}[series=MyList,label= ({\roman*})]
		\item  There exist $d=\operatorname{rank}(B)$ continuous functions $\gamma_1, \ldots, \gamma_d: \mathbb{R} \rightarrow$ $\mathbb{R} \medcup\{-\infty, \infty\}$ such that \[\mathcal{A}_c=\bigcup_{l=1}^d\left\{\gamma_{\ell}(\omega)+i \omega : \omega \in \mathbb{R}, \gamma_{\ell}(\omega) \notin\{-\infty, \infty\}\right\},\] which are called the spectral curves of $\mathcal{A}_c$.
		
		\item  Let $ \omega \in \mathcal{W}$, defined in \eqref{set_1}. Then,  there exist $l \in\{1, \ldots, d\}$ such that $\gamma_{\ell}(\omega)=\infty$ if and only if $i\omega \in \sigma(A)$.
		
		\item  Let $ i\omega \notin \sigma(A)$, $\rho\omega\neq k\pi\text{ for } k\in \mathbb{Z}\setminus\{0\}$, and $d<N$. Then,  there exist $l \in\{1, \ldots, d\}$ such that $\gamma_{\ell}(\omega)=$ $-\infty$ if and only if $i \omega \in \sigma\left(A_1\right)$.
	\end{enumerate}
\end{theorem}
\begin{proof}
	(i) It follows by Theorem \ref{Thm_App_1} that there exist $d$ continuous functions $x_{\ell}: \mathbb{R} \backslash \mathcal{W} \rightarrow \mathbb{C}$
	such that $p_\omega\left(x_{\ell}(\omega)\right)=0$ for $1 \leq l \leq d$.
	Define $\gamma_{\ell}(\omega):=-\log \left|x_{\ell}(\omega)\right|$. Then, there exist $d$ spectral curves $\gamma_{\ell}(\omega) $ that  satisfy
	\[p_\omega\left(e^{-(\gamma_{\ell}(\omega)+i\varphi)}\right)=0,\qquad \text{for some} \,  \varphi\in\mathbb{R},\]
	which can be extended continuously onto $\mathbb{R}$ with values in $\mathbb{R} \medcup\{-\infty,\infty\}$.\\
	
	(ii) Note that for all $\omega\in\mathcal{W}$  we have 
	\[i\omega\in \sigma(A)\Leftrightarrow
	p_{\omega}(0)=0\Leftrightarrow
	e^{-(\gamma_{\ell}(\omega)+i\varphi)}=0
	\Leftrightarrow
	\gamma_{\ell}(\omega)=\infty
	\text{ for some } \ell\in\{1,\ldots,d\}.
	\]
	
	(iii) For $x\neq0$, define 
	\[
	q_\omega(x)= \det\left(x(i \omega I_n-A)-\,\sinc( \rho\omega)\,B\right).
	\]
	Then, $q_\omega(x)=0$ is equivalent to $p_\omega(1/x)=0$ for $x\neq0$.  For $\ell\in\{1,\ldots,d\}$, the spectral curves $\gamma_{\ell}(\omega)$ satisfy $p_\omega\left(e^{-\gamma_{\ell}(\omega)} e^{-i \varphi}\right)=0$ and $q_\omega\left(e^{\gamma_{\ell}(\omega)} e^{i \varphi}\right)=0$ for $\gamma_{\ell}(\omega) \notin\{-\infty, \infty\}$ with some $\varphi \in \mathbb{R}$. We now study roots of $q_\omega(x)$ that approach zero.

	Using  $\sinc(\rho\omega)\ne 0$, the invertibility of $\widebar{B}$, and the same arguments as in \cite[Lemma 8]{lichtner2011spectrum}, we have that for all $\omega \in \mathbb{R}\backslash \{k\pi/\rho : k\in \mathbb{Z}\setminus\{0\}\}$, there exists $0<\delta\ll 1$ such that for sufficiently small $|x|<\delta$, we can write $q_\omega$ as
	\[
	q_\omega(x)=x^{N-d}\, \det \left(C_\omega(x) \right)\,\widebar{q}_\omega(x) ,
	\]
	where
	\[
	\widebar{q}_\omega(x)=\det\left(i \omega I_{N-d}-A_1 -x A_2 C(x)^{-1} A_3\right)
	\qquad\text{and} \qquad
	C_\omega(x)=x\left(i \omega I_d-A_4\right)-\sinc( \rho\omega)\,\widebar{B}.
	\]
	By assumption, $ i\omega \notin \sigma(A)$ and $\rho\omega\neq k\pi\text{ for } k\in \mathbb{Z}\setminus\{0\}$, which implies $\omega\notin{\mathcal A}_0$ and $p_\omega\not\equiv 0$. Thus we have ${q}_\omega(x)\not\equiv 0$, and hence, $\widebar{q}_\omega(x)\not\equiv 0$. From this expression, the root $x=0$ of $q_\omega$ is of multiplicity $N-d$. Since rank$(B)=d$, we know that $\widebar{q}_\omega(x)$ is a nontrivial component  of $q_\omega(x)$  for $|x|<\delta$. Furthermore, we have
	\[\gamma_{\ell}(\omega)=-\infty\Leftrightarrow
	\widebar{q}_\omega(0)=0\Leftrightarrow
	i\omega\in \sigma(A_1),\qquad \text{for some } \ell\in\{1,\ldots,d\}.
	\]
	
\end{proof}

Before presenting the main result on the convergence of the pseudo-continuous spectrum, we establish the following three lemmas.
\begin{lemma}\label{lemma_A_1}
	Let $i \omega_0 \in \mathcal{A}_0$. Then 
	\begin{enumerate}[series=MyList,label= ({\roman*})]
		\item  There exist $l \in \mathbb{N}$ and a polynomial $\widehat{p}_y(x)$ in $x$ that is nontrivial for all $y$ belonging to an open neighborhood $U \subset \mathbb{C}$ of $i \omega_0$ which contains no zero of $\sinhc(\rho y)$, such that
		\[
		\det(y\, I_n-A- x\,\sinhc(\rho y)\,B)=\left(y-i \omega_0\right)^l\,\widehat{p}_y(x).
		\]
		In particular 
		\[
		p_\omega(x)=i^l\,\left(\omega-\omega_0\right)^l\, \widehat{p}_{i \omega}(x)
		\]
		where $\widehat{p}_{i \omega}(x)$ is a nontrivial polynomial in $x$ for $\omega$ in some open neighborhood $\Omega$ of $\omega_0$ which contains no zero of $\sinhc(\rho y)$.

		\item  There exist $m \in \mathbb{N}$ and a polynomial $\widehat{q}_z(x)$ in $x$ that is nontrivial for $z$ belonging to an open neighborhood $V \subset \mathbb{C}$ of $i \omega_0$ which contains no zero of $\sinhc(\rho y)$, such that
		\[
		\det(x\, (z I_n-A)-\sinhc(\rho y)\,B)=\left(z-i \omega_0\right)^m \widehat{q}_z(x),
		\]
		In particular 
		\[
		q_\omega(x)=i^m \,\left(\omega-\omega_0\right)^m \, \widehat{q}_{i \omega}(x)
		\]
		where $\widehat{q}_{i \omega}(x)$ is a nontrivial polynomial in $x$ for $\omega$ in some open neighborhood $\widebar{\Omega}$ of $\omega_0$  which contains no zero of $\sinhc(\rho y)$.
	\end{enumerate}
\end{lemma}
\begin{proof}
	Define
	\[F(x,y):=\det(y\,I_n-A- x\,\sinhc(\rho y)\,B).\]
	\noindent  (i) If $i\omega_0\in\mathcal{A}_0$ and $\sinc(\rho\omega_0)\ne 0$, then either $A_2=0$ or $A_3=0$ and $i\omega_0\in\sigma(A_1)$. Consequently, we have
	\[
	F(x,y)=\det(y\,I_{n-d}-A_1)\det(y\,I_d - A_4 - x\,\sinhc(\rho y)\,\widebar{B})
	\]
	Hence,  $i\omega_0$ is a root of the first factor and the maximal power of $y-i\omega_0$ can be factored out:
	\[
	\det(y\,I_{n-d}-A_1)= (y-i\omega_0)^l p_1(y)
	\]
	where $p(y)$ is a polynomial in $y$ which is nontrivial near $i\omega_0$. The second factor is a polynomial in $y$ and $x\,\sinhc(\rho y)$: 
	\[
	\det(y\,I_d - A_4 - x\,\sinhc(\rho y)\,\widebar{B})=p_2(x\,\sinhc(\rho y),y)
	\]
	The leading term is $x^d$ which has coefficient $(-\sinhc(\rho y))^d\det{(\widebar{B})}$. Since $\widebar{B}$ is invertible, this is nonzero for $y$ away from zeros of $\sinhc(\rho y)$.
	Thus, we have
	\[
	F(x,y)=(y-i\omega)^l\,\widehat{p}_y(x)
	\]
	where $\widehat{p}_y(x)=p_1(y)p_2(x\,\sinhc(\rho y),y)$ is a polynomial in $x$ where the coefficients are products of powers of $y$ and $\sinhc(\rho y).$
	
	If $i\omega_0\in\mathcal{A}_0$ and $\sinc(\rho\omega_0)= 0$, then $i\omega_0\in\sigma(A)$. Hence, $i\omega_0$  is a root of the equation $F(x,y)=0$. Since
	\[F(x,i\omega_0)=\det(i\omega_0\,I_n-A- x\,\sinhc(\rho i\omega_0)\,B)=\det(i\omega_0\,I_n-A)=0\]
	and $F$ is holomorphic in $y$ for every $x$, then Weierstrass preparation theorem \cite{fuks1963theory}  implies that there exist an integer $l> 0$ and a holomorphic function $G(x, y)$ with $G(x, i \omega) \neq 0$ for every $x$ such that $F(x, y)=(y-i \omega)^l G(x, y)
	$. Consequently, we define $\hat{p}_y(x)=G(x, y)$.

	\noindent   (ii) Let $i\omega_0\in\mathcal{A}_0$. Following the same arguments as in (i) when    $\sinc(\rho\omega_0)\ne 0$, we have
	\begin{align*}
		\det(x\, (z I_n-A)-\sinhc(\rho z)\,B)&=\det(x(z\,I_{n-d}-A_1))\,\det(x(z\,I_d - A_4)- \,\sinhc(\rho z)\,\widebar{B})\\
		&= x^d\,\det(z\,I_{n-d}-A_1)\,\det(x(z\,I_d - A_4)- \,\sinhc(\rho z)\,\widebar{B})\\
		&= (z-i\omega_0)^m\, x^d\, q_1(z)\, q_2(\sinhc(\rho z),x\,z):=(z-i\omega_0)^m\,\widehat{q}_z(x).
	\end{align*}
	When    $\sinc(\rho\omega_0)= 0$, then $\det(x\, (i\omega_0 I_n-A))=0$ for every $x$; hence,   there exists a holomorphic function $H(x, y)$ with $H(x, i \omega) \neq 0$ for every $x$ such that 
	\begin{align*}
		\det(x\, (z I_n-A)-\sinhc(\rho z)\,B) 
		&= (z-i\omega_0)^m\, H(x,y):=(z-i\omega_0)^m\,\widehat{q}_z(x).
	\end{align*}
\end{proof}

Recall the asymptotic continuous spectrum, $\mathcal{A}_c$, defined in equation~\eqref{eqq_A0_new} and the scaling $S_\epsilon$ defined in equation~\eqref{Sepsilon}. 
\begin{lemma}\label{lemma_A_2}
	Let $\left\{\lambda_m\right\}_{m \in \mathbb{N}}$ be a sequence of complex numbers converging to $i \omega_0 \in \mathbb{C}$, where $\omega_0 \in \mathbb{R}$, and let $\left\{\epsilon_m\right\}_{m \in \mathbb{N}}$, with $0<\epsilon_m<1/\rho$
	and $\lim\limits_{m \rightarrow \infty} \epsilon_m=0$, be such that $\Delta\left(\lambda_m; \epsilon_m\right)=0$. Then, there exists a subsequence $\left\{\lambda_{m_k}\right\}_{k \in \mathbb{N}}\subset \left\{\lambda_m\right\}_{m \in \mathbb{N}}$ such that one of the following holds:\\
	\begin{enumerate}[series=MyList,label= ({\roman*})]
		\item $\lim\limits_{k \rightarrow \infty} S_{\epsilon_{m_k}}\left(\lambda_{m_k}\right) \in \mathcal{A}_c$;
		\item  $\lim\limits_{k \rightarrow \infty} S_{\epsilon_{m_k}}\left(\lambda_{m_k}\right)=\infty$ and there exists a spectral curve $\gamma_l$ with $\gamma_l\left(\omega_0\right)=\infty$, $l \in\{1, \ldots, d\}$;
		\item $\lim\limits_{k \rightarrow \infty} S_{\epsilon_{m_k}}\left(\lambda_{m_k}\right)=-\infty$ and there exists a spectral curve $\gamma_l$ with $\gamma_l\left(\omega_0\right)=-\infty$, $l \in\{1, \ldots, d\}$;
		\item  $\lambda_{m_k}=\lambda_{m_{k+1}}=i \omega_0 \in \mathcal{A}_0$ for all $k \in \mathbb{N}$.
	\end{enumerate}   
\end{lemma}
\begin{proof}
	Write $\lambda_m=\epsilon_m \gamma_m+i \epsilon_m \varphi_m+i 2 \pi\,\epsilon_m  q_m$ where $\varphi_m \in[0,2 \pi)$ and $q_m \in \mathbb{Z}$. By assumption we have $\lim\limits _{n \rightarrow \infty}\left(\epsilon_m \gamma_m\right)=\lim\limits _{n \rightarrow \infty}\left(\epsilon_m \varphi_m\right)=0$ and  $\lim\limits _{m \rightarrow \infty}\left(i 2 \pi\,\epsilon_m  q_m\right)=i \omega_0$.
	Since $\left\{\varphi_m\right\}_{m \in \mathbb{N}}$ is a bounded sequence in $\mathbb{R}$, it has a convergent subsequence. 
	To simplify notation, we may pass to that subsequence and assume that $\lim\limits _{m \rightarrow \infty} \varphi_m=$ $\varphi_0 \in \mathbb{R}$. Define
	\[
	\kappa_m(y):=\det\left((\epsilon_m y+i \epsilon_m 2 \pi q_m) I_n-A-\sinhc(\rho (\epsilon_m y+i \epsilon_m 2 \pi q_m))e^{-y}\,B \right), \quad y \in \mathbb{C}.
	\]
	The sequence $\kappa_m(y)$ of holomorphic functions converges uniformly to $p_{\omega_0}\left(e^{-y}\right)$  on bounded sets of $\mathbb{C}$. Moreover, it follows from $e^{-i(\varphi_m+2\pi q_m)}=e^{-i\varphi_m}$ that 
	\begin{equation}\label{eqq_99}
		\Delta(\lambda_m;\epsilon_m)=\kappa_m\left(\gamma_m+i \varphi_m\right)
		=p_{\epsilon_m2\pi q_m}
		\left(e^{-(\gamma_m+i\varphi_m)}\right)+O\left(\epsilon_m(\gamma_m+i\varphi_m)\right)
		=0.
	\end{equation}
	Suppose $\gamma_m$ is bounded. Then, there exists a subsequence $\left\{\gamma_{m_k}\right\}_{k \in \mathbb{N}}$ converging to some $\gamma_0 \in \mathbb{R}$  and $p_{\omega_0}\left(e^{-\gamma_0-i \varphi_0}\right)=0$ as $k \rightarrow \infty$  in \eqref{eqq_99}. 
	Suppose $i \omega_0 \notin \mathcal{A}_0$. It follows by the definition of $\mathcal{A}_c$ in \eqref{eqq_A0_new} that 
	\[\lim\limits_{k \rightarrow \infty} S_{\epsilon_{m_k}}\left(\lambda_{m_k}\right)=\lim\limits_{k \rightarrow \infty} \left(\gamma_{m_k}+i \epsilon_{m_k} \varphi_{m_k}+i 2 \pi\,\epsilon_{m_k}  q_{m_k}\right)=\gamma_0+i \omega_0 \in \mathcal{A}_c\] 
	Thus (i) holds in this case. 
	Assume $i \omega_0 \in \mathcal{A}_0$. Then, 
	Lemma \ref{lemma_A_1} implies that in an open neighborhood $U \subset \mathbb{C}$ of $i \omega_0$ we have, for each $k\in\mathbb{N}$,
	\[\det (\lambda_{m_k} I_n - A - e^{-\gamma_{m_k}-i\varphi_{m_k}}\sinhc(\rho\,\lambda_{m_k})B)
	=\left(\lambda_{m_k}-i \omega_0\right)^{l}\,\widehat{p}_{\lambda_{m_k}}\left(e^{-\gamma_{m_k}-i\varphi_{m_k}}\right)=0\]
	for 
	a nontrivial polynomial $\widehat{p}_{\lambda_{m_k}}(x)$  in $U $.

	Hence,  
	\begin{equation}\label{eeqq_11}
		\lambda_{m_k}=i \omega_0 \quad \text { or } \quad \widehat{p}_{\lambda_{m_k}}\left(e^{-\gamma_{m_k}} e^{-i \varphi_{m_k}}\right)=0.   
	\end{equation}
	Note that (iv) holds whenever  the set $\left\{k \in \mathbb{N} \,:\, \lambda_{m_k}=i \omega_0\right\}$ is infinite.
	Otherwise,  $\lambda_{m_k} \neq i \omega_0$ for all $k$ sufficiently large, and we obtain from \eqref{eeqq_11} that
	\[
	\widehat{p}_{ i \omega_0}\left(e^{-\gamma_0-i \varphi_0}\right)=\lim\limits _{k \rightarrow \infty} \widehat{p}_{\lambda_{m_k}}\left(e^{-\gamma_{m_k}} e^{-i \varphi_{m_k}}\right)=0
	\]
	Since $\widehat{p}_{\lambda_{m k}}(x)$ is nontrivial for values of $\omega$ near $\omega_0$, it follows that $\widehat{p}_{i \omega_0}(x)$ is likewise nontrivial. Hence, $p_\omega(x)$ has a root $x_l(\omega)$ which converges to $e^{-\gamma_0-i \varphi_0}$ for $\omega \rightarrow \omega_0$. 
	Consequently, the corresponding spectral curve $\gamma_l(\omega)=-\log \left|x_l(\omega)\right|$ satisfies $\gamma_l\left(\omega_0\right)=\gamma_0$. Thus,  $\gamma_0+i \omega_0 \in \mathcal{A}_c$ and (i) holds.

	\noindent Now, suppose  $\gamma_m$ is unbounded. We consider four cases:
	\begin{enumerate}[series=MyList, label=Case~\arabic*., leftmargin=2cm]
		\item $\lim\limits_{k \rightarrow \infty} \gamma_{m_k}= \infty$ and $i \omega_0 \notin \mathcal{A}_0$: 
		Taking $k \rightarrow \infty$ in equation \eqref{eqq_99} implies that $p_{\omega_0}(0)=0$. Since $i \omega_0 \notin \mathcal{A}_0$, there exists $l \in\{1, \ldots, d\}$ such that $\gamma_l\left(\omega_0\right)=\infty$; and hence, (ii) holds.

		\item $\lim\limits_{k \rightarrow \infty} \gamma_{m_k}= \infty$ and $i \omega_0 \in \mathcal{A}_0$: 
		Using \eqref{eeqq_11}, we have that either (iv) holds, or 
		we assume $\lambda_{m_k} \neq i \omega_0$ and $\widehat{p}_{\lambda_{m_k}}\left(e^{-\gamma_{m_k}} e^{-i \varphi_{m_k}}\right)=0$ for all $k\in\mathbb{N}$.
		In the latter case, $\widehat{p}_{i\omega_0}(0)=0$, which implies (ii).

		\item $\lim\limits_{k \rightarrow \infty} \gamma_{m_k}= -\infty$ and $i \omega_0 \notin \mathcal{A}_0$: Define
		\[
		\varrho_m(y):=\operatorname{det}\left(e^y\left((\epsilon_m y+i \epsilon_m 2 \pi q_m) I_n-A\right)-\sinhc(\rho (\epsilon_m y+i \epsilon_m 2 \pi q_m))\,B\right), \quad y \in \mathbb{C} .
		\]
		The sequence $\varrho_m(y)$ of holomorphic functions converges uniformly to $q_{\omega_0}\left(e^y\right)$ on bounded sets of $\mathbb{C}$. Moreover, we have
		\[
		\varrho_m\left(\gamma_m+i \varphi_m\right)=\left(e^{\gamma_m+i \varphi_m}\right)^n \Delta\left(\lambda_m; \epsilon_m\right)=0
		\]
		Letting $k \rightarrow \infty$, we have $q_{\omega_0}(0)=0$. Since $i \omega_0 \notin \mathcal{A}_0$, the function $q_{\omega_0}(x)$ is nontrivial  for sufficiently small $|x|$, see Theorem \ref{theorem_rescaled_pseudo_continuous}. Then, there exists $l \in\{1, \ldots, d\}$ such that $\gamma_l\left(\omega_0\right)=-\infty$; and hence, (iii) holds.

		\item $\lim\limits_{k \rightarrow \infty} \gamma_{m_k}= -\infty$ and $i \omega_0 \in \mathcal{A}_0$: Applying Lemma \ref{lemma_A_1} leads to
		\[
		\lambda_{m_k}=i \omega_0 \quad \text { or } \quad \widehat{q}_{\lambda_{m_k}}\left(e^{\gamma_{m_k}} e^{i \varphi_{m_k}}\right)=0,
		\]
		which implies that either (iv) or (iii) holds using the same arguments as above.
	\end{enumerate}
\end{proof}

\begin{lemma}\label{lemma_A_3}
	For $m \in \mathbb{N}$, let $0<\epsilon_m\le1/\rho$ be such that 
	$\lim\limits_{m \rightarrow \infty} \epsilon_m=0$. 
	Let $R>0$ and consider the sequence $\left\{\lambda_{m}\right\}_{m \in \mathbb{N}} \subset \mathcal{P}_c^{\epsilon_m}$ with $\left|{\rm Im} (\lambda_m)\right| \leq R$. 
	Then, 
	\begin{enumerate}[series=MyList,label= ({\roman*})]
		\item $\left\{\lambda_{m}\right\}_{m \in \mathbb{N}}$ is bounded
		\item  For any convergent subsequence $\left\{\lambda_{m_k}\right\}_{k \in \mathbb{N}}\subset \left\{\lambda_{m}\right\}_{m \in \mathbb{N}}$, we have $\lim\limits_{k \rightarrow \infty} {\rm Re}\left\{\lambda_{m_k}\right\}_{k \in \mathbb{N}}=0$.
	\end{enumerate} 
\end{lemma}
\begin{proof}
	(i) Let $\left\{\lambda_{m}\right\}_{m \in \mathbb{N}} \subset \mathcal{P}_c^{\epsilon_m}$ with $\left|{\rm Im} (\lambda_m)\right| \leq R$. 
	By contradiction, suppose that $\left\{\lambda_{m}\right\}_{m \in \mathbb{N}}$ is unbounded. Then, there exists a subsequence $\left\{\lambda_{q_k}\right\}_{k \in \mathbb{N}}\subset \left\{\lambda_{m}\right\}_{m \in \mathbb{N}}$ such that either
	\begin{enumerate}[series=MyList, label=Case~\arabic*., leftmargin=2cm]
		\item $\lim\limits_{k \rightarrow \infty} {\rm Re}\left\{\lambda_{q_k}\right\}_{k \in \mathbb{N}}=\infty$, or
		\item $\lim\limits_{k \rightarrow \infty} {\rm Re}\left\{\lambda_{q_k}\right\}_{k \in \mathbb{N}}=-\infty$.
	\end{enumerate}
	Case 1 contradicts the boundedness of the spectral radius of
	\[A+B\,e^{-\lambda_{q_k}/\epsilon_{q_k}}\,\sinhc(\lambda_{q_k}\rho)
	= A+B \frac{e^{-\lambda_{q_k}/\epsilon_{q_k}\,(1-\epsilon_{q_k}\rho)}-e^{-\lambda_{q_k}/\epsilon_{q_k}\,(1+\epsilon_{q_k}\rho)}}{2\lambda_{q_k}\rho}.
	\]
	For case 2, similar to Theorem \ref{Thm_App_1}, we use \cite[Lemma 7]{lichtner2011spectrum} and obtain that the leading term of 
	\[
	\Delta(\lambda_{q_k};\epsilon_{q_k})=\det\left(\lambda_{q_k} I_n-A-B\,e^{-\lambda_{q_k}/\epsilon_{q_k}}\,\sinhc(\lambda_{q_k}\rho)\right)
	\]
	is
	\[
	\Lambda_m:=  \left( e^{-\lambda_{q_k}/\epsilon_{q_k}} \right)^d \sinhc^d(\rho\lambda_{q_k})\,\det\left(\lambda_{q_k} I_{n-d}-A_1\right) \prod_{j=1}^d \widebar{\lambda}_j
	\]
	where  $\widebar{\lambda}_j\neq0$ for all $j=1,\ldots,d$ are the eigenvalues of the invertible matrix $\widebar{B}$. Hence, $\Lambda_m\to \pm \infty$ depending on  the value of $n-d$. This implies that $\lambda_m \notin \mathcal{P}_c^{\epsilon_m}$, a contradiction.

	(ii)  Let $\left\{\lambda_{m_k}\right\}_{k \in \mathbb{N}}\subset \mathcal{P}_c^{\epsilon_{m_k}}$ be a subsequence converging to $\lambda_0$.
	If ${\rm Re}(\lambda_0)>0$, we obtain from  $\Delta(\lambda_{m_k};\epsilon_{m_k})=0$ that $\det(\lambda_0 I_n-A)=0$ as $k\to \infty$. Thus, $\lambda_0\in \mathcal{A}_{+}$, and hence,  a tail of the subsequence  $\left\{\lambda_{m_k}\right\}_{k \in \mathbb{N}}$ is in $\mathcal{P}_+^{\epsilon_m}$, which contradicts that fact that $\left\{\lambda_{m_k}\right\}_{k \in \mathbb{N}}\subset \mathcal{P}_c^{\epsilon_m}$.

	Suppose ${\rm Re}(\lambda_0)<0$. Then there exists a vector $v^{(k)}=[v^{(k)}_{n-d},v^{(k)}_{n}]^T \in \mathbb{C}^n$ such that $\left\|v^{(k)}\right\|=1$ and
	\begin{equation}\label{eq_444}
		\begin{bmatrix}
			\lambda_{m_k}  I_{n-d} - A_1 &  -A_2 \\
			-A_3 & \lambda_{m_k}  I_d - A_4 - \widebar{B}\,e^{-\lambda_{m_k}/\epsilon_{m_k}}\,\sinhc(\lambda_{m_k}\rho)
		\end{bmatrix}
		\,
		\begin{bmatrix}
			v^{(k)}_{n-d} \\
			v^{(k)}_{d}
		\end{bmatrix}=\begin{bmatrix}
			0_{n-d} \\
			0_{d}
		\end{bmatrix}  
	\end{equation}
	Multiplying both sides of equation \eqref{eq_444} with $e^{\lambda_{m_k}/\epsilon_{m_k}} \cschc(\lambda_{m_k}\rho)$, where $\cschc(x)=1/\sinhc(x)$, gives
	\[
	\begin{bmatrix}
		(\lambda_{m_k}  I_{n-d} - A_1) e^{\lambda_{m_k}/\epsilon_{m_k}} \cschc(\lambda_{m_k}\rho) &  -A_2 e^{\lambda_{m_k}/\epsilon_{m_k}} \cschc(\lambda_{m_k}\rho)\\
		-A_3 e^{\lambda_{m_k}/\epsilon_{m_k}} \cschc(\lambda_{m_k}\rho) & e^{\lambda_{m_k}/\epsilon_{m_k}}\cschc(\lambda_{m_k}\rho) (\lambda_{m_k}  I_d - A_4) - \widebar{B}
	\end{bmatrix}
	\,
	\begin{bmatrix}
		v^{(k)}_{n-d} \\
		v^{(k)}_{n}
	\end{bmatrix}=\begin{bmatrix}
		0_{n-d} \\
		0_{d}
	\end{bmatrix}
	\]
	As $k\to \infty$ and due to the fact that ${\rm Re}(\lambda_0)<0$, we have 
	\[
	\begin{bmatrix}
		0 &  0 \\
		0 & - \widebar{B}
	\end{bmatrix}
	\,
	\begin{bmatrix}
		\widebar{v}_{n-d} \\
		\widebar{v}_{n}
	\end{bmatrix}=\begin{bmatrix}
		0_{n-d} \\
		0_{d}
	\end{bmatrix}
	\]
	where $\widebar{v}_{s}=\lim\limits_{k \rightarrow \infty} v^{(k)}_{s}$, $s\in\{n-d,d\}$ with $\widebar{v}=\lim\limits_{k \rightarrow \infty} v^{(k)}$
	Thus, $\widebar{B} \widebar{v}_{d}=0_{d}$, and hence, $\widebar{v}_{d}=0_{d}$ due to the invertibility of $\widebar{B}$. 
	Thus we must have $\widebar{v}_{n-d}\neq 0_{n-d}$. Otherwise this contradicts  $\left\|\widebar{v}\right\|=1$. 
	Then, equation \eqref{eq_444} implies that $(\lambda_{0}  I_{n-d} - A_1)\widebar{v}_{n-d}=0_{n-d}$ with $\widebar{v}_{n-d}\neq 0_{n-d}$. Therefore,  $\lambda_0\in \mathcal{A}_{-}$, and hence,  a tail of the subsequence  $\left\{\lambda_{m_k}\right\}_{k \in \mathbb{N}}$ is in $\mathcal{P}_-^{\epsilon_m}$, which contradicts the fact that $\left\{\lambda_{m_k}\right\}_{k \in \mathbb{N}}\subset \mathcal{P}_c^{\epsilon_m}$.
\end{proof}

Now, we establish that as $\epsilon \to 0$, the rescaled pseudo-continuous spectrum $\pieps_c$ approaches the set of curves defined by the asymptotic continuous spectrum $\mathcal{A}_c$.

\begin{theorem}\label{thm_append_2}
	
	\begin{enumerate}[series=MyList,label= ({\roman*})]
		\item For any $\mu \in \mathcal{A}_c$ and $\delta > 0$, there exists $0<\epsilon_0\le 1/\rho$  so that whenever $0 < \epsilon < \epsilon_0$, there exists $\lambda \in \sigmaeps_c$ satisfying $\left|\seps(\lambda) - \mu\right| < \delta$.

		\item Let $R > 0$ and $r$ be defined by \eqref{r_def}. For any $0 < \delta \leq r$, there exists $0<\epsilon_0\le 1/\rho$ so that whenever $0 < \epsilon < \epsilon_0$ and $\lambda \in \sigmaeps_c$ with $|{\rm Im}(\lambda)| < R$, we have $|{\rm Re}(\lambda)| < \delta$ and there exists $\mu \in \mathcal{A}_c$ such that $\left|\seps(\lambda) - \mu\right| < \delta$.
	\end{enumerate}   
\end{theorem}
\begin{proof}
	(i) Let $\mu=\gamma_0+i \omega_0 \in \mathcal{A}_c$. Without loss of generality we assume that $i \omega_0 \notin \mathcal{A}_0$. Let $0<\epsilon\ll 1$ and define
	\[
	\Psi_\epsilon(z):=\det\left(\left(\epsilon z+i  2 \epsilon \pi k(\epsilon)\right)\,I_n-A-B\,\sinhc(\left(\epsilon z+i  2 \epsilon \pi k(\epsilon)\right)\rho)\,e^{-z}\right)
	\]
	where $k(\epsilon):=\left[\frac{\omega_0}{ 2 \epsilon \pi}\right]$ is the largest integer smaller than $\frac{\omega_0}{ 2 \epsilon \pi}$. 
	As $\epsilon \to 0$, we obtain
	\[
	\Psi_\epsilon(z)\rightarrow\Psi_0(z)=\det\left(i\omega_0 \,I_n-A-B\,\sinc(\omega_0\rho)\,e^{-z}\right)
	\]
	locally uniformly on $\mathbb{C}$. 
	Since $i \omega_0 \notin \mathcal{A}_0$, the function $\Psi_0$ is nontrivial. By assumption there exists $\varphi_0 \in \mathbb{R}$ such that $\Psi_0(z_0)=0$ where $z_0=\gamma_0+i \varphi_0$. 
	Choose $0<\eta\ll 1$ such that 
	\[
	\widehat{\ball}_{\eta}(z_0)=\left\{ z\,:\, |z-z_0|\le \eta\,\text{and}\, \Psi_0(z)=0  \right\}=\{z_0\}.
	\]
	Note that $\widehat{\ball}_{\eta}(z_0)$ contains only the root $z_0$ of $\Psi_0(z)=0$. Using Hurwitz's theorem \cite{gamelin2003complex} and the same arguments as in Theorem \ref{thm_append_1}, there exists a sufficiently small $0<\epsilon<\epsilon_0\le 1/\rho$ such that $\Psi_0$ and $\Psi_\epsilon$ have the same number of zeros, counting multiplicity, in $\ball_{\eta}(z_0)$.
	Suppose  $z_\epsilon=\gamma_\epsilon+i\varphi_\epsilon$ is such a zero. By Hurwitz's Theorem, $z_\epsilon\rightarrow z_0$. Then, $\lambda_\epsilon=\epsilon z_\epsilon+i  2 \epsilon \pi k(\epsilon) \in \mathcal{P}_c^\epsilon$ because  
	$e^{-z_{\epsilon}}=e^{-\lambda_{\epsilon}/\epsilon}$ and $\lambda_\epsilon\rightarrow i\omega_0$ as $\epsilon\rightarrow 0$. 
	Note that $S_\epsilon\left(\lambda_\epsilon\right)=\gamma_\epsilon+i\epsilon \varphi_\epsilon+i  2 \epsilon \pi k(\epsilon)$. 
	Let $\delta>0$. Then we can choose $\eta>0$ and $0<\epsilon<\epsilon_0\le 1/\rho$ sufficiently small such that $\left|S_\epsilon\left(\lambda_\epsilon\right)-\mu\right|<\delta$.

	(ii) Assume for contradiction that there exist $R_0 > 0$ and $\delta_0>0$ such that for every $0<\epsilon<\epsilon_0\le 1/\rho$, there exists $\epsilon \in\left(0, \epsilon_0\right)$ and $\lambda \in \mathcal{P}_c^{\epsilon}$ with $|{\rm Im}(\lambda)|<R_0$ for which either
	$|{\rm Re}(\lambda)| \geq \delta_0$ or $\left|S_\epsilon(\lambda)-\mu\right|\geq\delta_0$  for all $\mu \in A_c$. 
	Consequently, we construct a positive sequence $\epsilon_m \in\left(0, \frac{1}{m}\right)$ for $m \in \mathbb{N}$ such that $\lambda_m \in \mathcal{P}_c^{\epsilon_m}$  with $|{\rm Im}(\lambda_m)|<R_0$  for which either 
	\begin{enumerate}[series=MyList, label=Case~\arabic*., leftmargin=2cm]
		\item $|{\rm Re}(\lambda_m)| \geq \delta_0$  for all  $m \in \mathbb{N}$, or
		\item $\left|S_{\epsilon_m}(\lambda_m)-\mu\right|\geq\delta_0$  for all $\mu \in \mathcal{A}_c$ and $m \in \mathbb{N}$.
	\end{enumerate}
	Note that $\lim\limits_{m\to\infty}\epsilon_m=0$. Case 1 contradicts Lemma \ref{lemma_A_3} that the sequence $\lambda_m$ has a convergent subsequence $\left\{\lambda_{q_k}\right\}_{k \in \mathbb{N}}$ such that $\lim\limits_{k \rightarrow \infty} {\rm Re}\left\{\lambda_{q_k}\right\}_{k \in \mathbb{N}}=0$.
	For case 2, note that the boundedness of the imaginary parts $\left|{\rm Im} (\lambda_m)\right| < R_0$ and Lemma \ref{lemma_A_3} imply that there exists a subsequence $\left\{\lambda_{m_k}\right\}_{k \in \mathbb{N}}$ converging to some $i \omega_0 \in \mathbb{C}$. Consequently, Lemma \ref{lemma_A_2} implies four cases: Lemma \ref{lemma_A_2}(i) provides that the rescaled subsequence $\left\{S_{\epsilon_{m_k}}\left(\lambda_{m_k}\right)\right\}_{k \in \mathbb{N}}$ is converging to $\mathcal{A}_c$; therefore, there exist $\mu\in $ and  $K\in\mathbb{N}$ such that  $\left|S_{\epsilon_{m_k}}(\lambda_{m_k})-\mu\right|<\delta_0$ for all $k\ge K$, which contradicts the assumption that  $\left|S_{\epsilon_m}(\lambda_m)-\mu\right|\geq\delta_0$; 
	Lemma \ref{lemma_A_2}(ii) and (iii) imply that the rescaled subsequence diverges to $\pm\infty$. Moreover, there exists a spectral curve $\gamma_l$, with $l \in \{1, \ldots, d\}$, in $\mathcal{A}_c$ that has a singularity at $i\omega_0$ in the sense that $\gamma_l(\omega_0)=\pm\infty$. This behavior implies convergence in $\mathcal{A}_c$, which contradicts the assumption that
	$\left|S_{\epsilon_m}(\lambda_m)-\mu\right|\geq\delta_0$  holds for all $\mu \in \mathcal{A}_c$;
	and 
	Lemma \ref{lemma_A_2}-(iv) provides that the subsequence stabilizes at $i \omega_0 \in \mathcal{A}_0$, that is  $\lambda_{m_k}=\lambda_{m_{k+1}}=i \omega_0$  for all $k \in \mathbb{N}$, which contradicts 
	$\lambda_m\in \sigmaeps_c=\sigmaeps \backslash \left(\sigmaeps_s  \medcup \mathcal{A}_0\right)$. 
\end{proof}

\begin{remark}\label{rem_rhodep}
	Theorem~\ref{thm_append_1} shows that for finite $\gavg$, large enough, each eigenvalue in the spectrum, $\mathcal{P}_s^\epsilon$, is well approximated by an eigenvalue of the asymptotic spectrum, ${\mathcal A}_s$. Similarly, each eigenvalue in the rescaled pseudo-continuous spectrum, $\mathcal{R}^\epsilon_c$, is well approximated by an eigenvalue of the asymptotic continuous spectrum, ${\mathcal A}_c$, by Theorem~\ref{thm_append_2}.
	In the application of these theorems to specific models, we will explore the effect of the half-width of the distribution, $\rho$, on the eigenvalue spectrum and the validity of these approximation. Clearly, how large $\gavg$ needs to be for these theorems to hold will depend on $\rho$. For example, the proof of Theorem~\ref{thm_append_1} requires that for fixed $\rho$, $\epsilon$ must be small enough that $\cschc((z+\lambda)\rho)=O(1)$ with respect to $\epsilon$.
\end{remark}

\end{document}